\documentclass[reqno,a4paper,11pt]{amsart}
\usepackage{fullpage,hyperref,cleveref,amssymb,mathtools,ifthen,bm,comment,soul,csquotes}
\usepackage{tikz}
\usepackage{tikz-3dplot}
\usepackage{lmodern}
\usepackage[shortlabels]{enumitem}
\usepackage[T1]{fontenc}
\usepackage[all]{hypcap}
\usepackage{array}
\usepackage{tikz}
\usepackage{wrapfig}
\usepackage{pgfplots}
\usepackage[normalem]{ulem}
\usepackage[export]{adjustbox}
\usetikzlibrary{patterns}
\usepackage{xcolor}
\usepackage{thm-restate}
\pgfplotsset{compat=1.18}

\usepackage[backend=biber,giveninits=true,doi=false,url=false,isbn=false,style=trad-plain,hyperref]{biblatex}
\bibliography{bibliography}

\newtheorem{theorem}{Theorem}[section]

\newtheorem{lemma}[theorem]{Lemma}
\newtheorem{corollary}[theorem]{Corollary}
\newtheorem{claim}[theorem]{Claim}
\Crefname{claim}{Claim}{Claims}
\theoremstyle{definition}

\newtheorem{remark}[theorem]{Remark}
\newtheorem{question}[theorem]{Question}

\newtheorem{property}{Property}
\renewcommand{\theproperty}{P\arabic{property}}
\Crefname{property}{Property}{Properties}
\creflabelformat{property}{#2(#1)#3}

\newcommand{\car}{\mathbin{\Box}}
\newcommand{\symdif}{\mathbin{\triangle}}

\newcommand{\Bin}{\operatorname{Bin}}
\newcommand{\Boot}{\operatorname{Boot}}
\newcommand{\dmin}{\delta(G)}

\newcommand{\dist}{\operatorname{dist}}
\newcommand{\supp}{\operatorname{supp}}
\newcommand{\Inv}{\operatorname{Inv}}

\renewcommand{\le}{\leqslant}
\renewcommand{\leq}{\leqslant}
\renewcommand{\ge}{\geqslant}
\renewcommand{\geq}{\geqslant}
\renewcommand{\emptyset}{\varnothing}
\newcommand{\eps}{\varepsilon}

\providecommand\given{\nonscript\:\ifthenelse{\equal{\delimsize}{}}{\big\vert}{\delimsize\vert}\nonscript\:\mathopen{}}
\let\Pr\undefined
\DeclarePairedDelimiterXPP\Pr[1]{\mathbb{P}}{[}{]}{}{#1}
\DeclarePairedDelimiterXPP\Ex[1]{\mathbb{E}}{[}{]}{}{#1}
\DeclarePairedDelimiterXPP\Prs[1]{\mathbb{P}^*}{[}{]}{}{#1}
\DeclarePairedDelimiterXPP\Exs[1]{\mathbb{E}^*}{[}{]}{}{#1}
\DeclarePairedDelimiterXPP\Prp[1]{\mathbb{P}'}{[}{]}{}{#1}

\title{Majority bootstrap percolation on the permutahedron and other high-dimensional graphs}
\author{Maur\'icio Collares \and Joshua Erde \and Anna Geisler \and Mihyun Kang}
\address{Institute of Discrete Mathematics, Graz University of Technology, Steyrergasse 30, 8010 Graz, Austria}
\email{mauricio@collares.org, \{erde, geisler, kang\}@math.tugraz.at}
\subjclass{60K35, 60C05 (Primary)}

\begin{document}
\begin{abstract}
Majority bootstrap percolation is a model of infection spreading in networks. Starting with a set of initially infected vertices, new vertices become infected once half of their neighbours are infected. Balogh, Bollob\'{a}s and Morris studied this process on the hypercube and showed that there is a phase transition as the density of the initially infected set increases. Generalising their results to a broad class of high-dimensional graphs, the authors of this work established similar bounds on the critical window, establishing a universal behaviour for these graphs.

These methods necessitated an exponential bound on the order of the graphs in terms of their degrees. In this paper, we consider a slightly more restrictive class of high-dimensional graphs, which nevertheless covers most examples considered previously. Under these stronger assumptions, we are able to show that this universal behaviour holds in graphs of \emph{superexponential order}. As a concrete and motivating example, we apply this result to the permutahedron, a symmetric high-dimensional graph of superexponential order which arises naturally in many areas of mathematics. Our methods also allow us to slightly improve the bounds on the critical window given in previous work, in particular in the case of the hypercube.

Finally, the upper and lower bounds on the critical window depend on the maximum and minimum degree of the graph, respectively, leading to much worse bounds for irregular graphs. We also analyse an explicit example of a high-dimensional \emph{irregular} graph, namely the Cartesian product of stars and determine the first two terms in the expansion of the critical probability, which in this case is determined by the minimum degree.
\end{abstract}
\maketitle

\section{Introduction}

\subsection{Motivation and background}

Bootstrap percolation is a discrete model for infection spreading over a network.
Starting with some set of initially infected vertices in a graph, new vertices are added to this set in a discrete sequence of time steps according to some local and homogenous rules.
Once infected, a vertex remains in this state forever, and so bootstrap percolation is an example of \emph{monotone cellular automata}.

The bootstrap percolation process was introduced in 1979 in the work on magnetic systems of Chalupa, Leath and Reich \cite{Chalupa1979Grid}.
It has found various applications in physics \cite{Morris2017Bootstrap}, for example to describe interacting particle systems \cite{JLMTZ17} or the zero-temperature Ising model \cite{Fontes2002Ising}.
Furthermore, bootstrap percolation is used to model processes in other areas, such as neural networks \cite{Amini2010Neural} or opinion spreading in sociology \cite{Granovetter1978Sociology}.
There are also connections to other processes on discrete structures, such as \emph{weak saturation} introduced by Bollob\'as \cite{Bollobas1968Saturation}, where edges of a graph $G$ get `activated' once they complete a copy of a fixed graph $H$.

The \emph{$r$-neighbour bootstrap percolation process} evolves in a sequence of time steps, called \emph{rounds}, where the set of infected vertices is updated in each round.
Starting with an initial set $A_0 \subseteq V(G)$ of infected vertices, in the $i$th round we add a vertex $v \in V(G)$ to $A_i$ if it is already infected, or at least a certain threshold $r$ of its neighbours are already infected, i.e., for $i \in \mathbb{N}$ we set
\[
A_{i} = A_{i-1} \cup \left\{ v \in V(G) \colon \left|N(v) \cap A_{i-1}\right| \geq r \right\}.
\]
This gives a non-decreasing sequence of sets $A_0 \subseteq A_1 \subseteq \ldots$, which is fully determined by the initial set $A_0$.
If there is an $i \in \mathbb{N}$ such that $A_i=V(G)$, the set $A_0$ is said to \emph{percolate} on $G$.
In other words, if we start with a percolating set, then the infection eventually spreads to the whole graph $G$.

In order to study the \emph{typical} behaviour of this process on a given graph $G$, the choice of $A_0$ is randomized, i.e., each vertex is added to $A_0$ initially with some probability $p$.
As the density of the initially infected set of vertices $A_0$ increases, the probability that the process percolates increases as well.
More precisely, for $p \in [0,1]$ let $\textbf{A}_p$ be a $p$-random subset of $V(G)$ where each vertex in $V(G)$ is included in $\textbf{A}_p$ independently with probability $p$.
Then we define
\[
  \Phi(p,G) \coloneqq \Pr{\textbf{A}_p \text{ percolates on } G}
\]
to be the probability that a $p$-random subset of $V(G)$ percolates. The critical probability is given by
\[
p_c(G) \coloneqq \inf \left\{ p \in (0,1) \colon  \Phi(p,G) \geq \frac{1}{2} \right\}.
\]
Throughout the paper we refer to a bootstrap percolation process initialised with a random set $A_0 \sim \textbf{A}_p$ as the \emph{random bootstrap percolation process.}

A series of results treated the critical probability of the $r$-neighbour bootstrap percolation process on grids \cite{AizenmanS2001Threshold,Balogh2003Sharp,Cerf1998Lattice3D,CerfManzo2002FinVol,GrHo08,HaMo19,holroyd2002sharp}, culminating in the ground-breaking result of Balogh, Bollobás, Duminil-Copin and Morris \cite{balogh2011Grid} who determined the sharp threshold for $r$-neighbour bootstrap percolation process on the $d$-dimensional grid for any $r, d \in \mathbb{N}$.
While many results treat the random bootstrap percolation process on bounded dimensional grids, considering the asymptotics as the side-lengths tend to infinity, there has also been interest in the asymptotics on grids of bounded side-length as the dimension tends to infinite.
Particularly well-studied is the `simplest' case of the $n$-dimensional hypercube.
Here, when $r=2$, Balogh and Bollobás \cite{BaBo06} determined a coarse threshold for $2$-neighbour bootstrap percolation on the hypercube and this was strengthened to a sharp threshold by Balogh, Bollobás and Morris \cite{BaBoMo2010HighDimGrid}, who also extended this result to other high-dimensional grids.

While the results mentioned above consider constant $r$, another well-studied instance of this process is the \emph{majority bootstrap percolation process}, where we set $r(v)=d(v)/2$ for every $v \in V(G)$.
In other words, a vertex becomes infected once at least half of its neighbours are infected.
Hence, at least heuristically, one should expect the threshold for majority bootstrap percolation to be close to $1/2$, and it is relatively easy to show that this is indeed true for $d$-regular graphs whose order is small (sub-exponential) in terms of $d$.
Balogh, Bollobás and Morris \cite{BaBoMo2009} showed that, under some mild expansion assumptions, this holds even for much larger graphs.

\begin{theorem}[{\cite[Theorem 2.2]{BaBoMo2009}}, informal]\label{thm:BaBoMogeneral}
  Let $G$ be a regular graph whose order is sufficiently small in terms of the neighbourhood expansion of the graph and consider the random majority bootstrap percolation process.
  Then
  \[
    p_c(G)=\frac{1}{2}+o(1).
  \]
\end{theorem}

This establishes a `sharp' threshold for a broad class of graphs, but leaves open the question of determining more precisely the critical probability and the width of the critical window.
In fact, in certain cases, it can be shown that the critical probability is in fact bounded away from $1/2$, in the sense that the points $p=1/2$ lies outside the \emph{critical window}.
In particular, in the case of the hypercube, Balogh, Bollob\'as and Morris \cite{BaBoMo2009} established sharp bounds on the critical probability up to the third term.

\begin{theorem}[{\cite[Theorem 2.1]{BaBoMo2009}}]\label{thm:BaBoMohypercube}
  Let $Q^n$ be the $n$-dimensional hypercube and consider the random majority bootstrap percolation process. Let
  \[
    p\coloneqq\frac{1}{2}-\frac{1}{2} \sqrt{\frac{\log n}{n}} + \lambda \frac{\log \log n}{\sqrt{n \log n}}.
  \]
  Then
  \[
    \lim_{n \to \infty} \Phi(p, Q^n) = \left\{
      \begin{array} {c@{\quad \textup{if} \quad}l}
        0 & \lambda \leq -2,       \\[+1ex]
        1 & \lambda > \frac{1}{2}.
      \end{array}\right.
  \]
\end{theorem}

The authors note \cite[Section 5]{BaBoMo2009} that the $\lambda \leq -2$ in \Cref{thm:BaBoMohypercube} is not sharp, and sketch how it could be improved to $\lambda < -\frac{3}{4}$ with a little extra care.
However, they note that, at least with their methods, there is a natural barrier at $\lambda = -\frac{1}{4}$.

Extending the work of \cite{BaBoMo2009}, the authors showed that the behaviour of the hypercube is in some sense \emph{universal} for a large class of high-dimensional geometric graphs \cite{CEGK24}.
Roughly, they showed that for a class of graphs $\mathcal{F}$ whose structure is in some sense close to be being defined by a local coordinate system, for a formal definition see \Cref{sec:definitionH}, the order of the second term in the expansion of the critical probability can be bounded above and below as a function of the maximum and minimum degree, respectively.

\begin{theorem}[{\cite[Theorem 1.4]{CEGK24}, informal}]\label{t:mainThm-classH}
  Let $(G_n)_{n \in \mathbb{N}}$ be a sequence of graphs in $\mathcal{F}$ where $\delta(G_n) \to \infty$ as $n \to \infty$ and consider the random majority bootstrap percolation process.
  Then for any constant $\eps > 0$,
  \[
    \lim_{n \to \infty} \Phi(p, G_n) = \left\{
      \begin{array} {c@{\quad \textup{if} \quad}l}
        0 & p < \frac{1}{2}-\left(\frac{1}{2}+\eps\right) \sqrt{\frac{\log \delta(G_n)}{\delta(G_n)}}, \\[+1ex]
        1 & p > \frac{1}{2}-\left(\frac{1}{2}-\eps\right) \sqrt{\frac{\log \Delta(G_n)}{\Delta(G_n)}}.
      \end{array}\right.
  \]
\end{theorem}

In particular, if the graph is $n$-regular, then this determines the second term in the expansion of the critical probability, which agrees with that given by \Cref{thm:BaBoMohypercube}.
Let us note however some shortcomings of \Cref{t:mainThm-classH}.
Firstly, we note that even for regular graphs, \Cref{t:mainThm-classH} is not quite as precise as \Cref{thm:BaBoMohypercube} and gives no bounds on the order of the third term in this expansion.
Furthermore, whilst there are some natural graphs of superexponential order which satisfy the structural properties determining the class $\mathcal{F}$, for technical reasons the order of the graphs in $\mathcal{F}$ are restricted to being exponential in their minimum degree.
Finally, for irregular graphs whose maximum and minimum degrees differ by a factor asymptotically strictly greater than $1$, \Cref{t:mainThm-classH} does not even determine the order of the second term in this expansion.

More recently, Zhu \cite{Zhu25} considered the majority bootstrap percolation process on the \emph{generalised hypercube} $Q_k^n$ on vertex set $\{0,1\}^n$ where two vertices are connected if their Hamming distance is at most $k \in \mathbb{N}$.
Note that this graph is regular with degree $d= \sum_{i=1}^k \binom{n}{i}$, but does not lie in $\mathcal{F}$.
Here, Zhu showed that the critical probability is bounded away from $1/2$, and their results leave open the possibility that $Q_k^n$ exhibits the same universal behaviour as in \Cref{t:mainThm-classH}.

\begin{theorem}[{\cite[Theorem 3]{Zhu25}}]
  For $k \geq 2$ let $Q^n_k$ be the generalised hypercube and consider the random majority bootstrap percolation process.
  Then there is a $\beta<1$ such that
  \[
    \lim_{n \to \infty} \Phi(p, Q^n_k) = \left\{
      \begin{array} {c@{\quad \textup{if} \quad}l}
        0 & p < \frac{1}{2}-\frac{1}{n^{\beta}}, \\[+1ex]
        1 & p > \frac{1}{2}-\frac{1}{n^{k/2}}.
      \end{array}\right.
  \]
\end{theorem}

\subsection{Our contribution}
Our aim with this paper is to address, to some extent, some of the shortcomings of \Cref{t:mainThm-classH}.

First, whilst the proof of the $1$-statement in \Cref{t:mainThm-classH} can be extended quite easily to graphs $G$ of order $\exp(\text{poly}(\Delta(G))$, and in fact gives the same bound on the third order term as in \Cref{thm:BaBoMohypercube} (see \cite[Remark 4.4]{CEGK24}), the proof of the $0$-statement crucially uses the fact that $|V(G)| \leq \exp(K \delta(G))$ for some constant $K\in \mathbb{N}$ for every $G \in \mathcal{F}$, and this restriction is also implicit in the proof of \Cref{thm:BaBoMohypercube}.

In order to overcome this restriction, we introduce another class $\mathcal{H}$ of high-dimensional graphs (see \Cref{sec:H'} for a formal definition) which satisfy slightly stronger structural assumptions than those in $\mathcal{F}$, but which nevertheless still in some sense capture the idea of a high-dimensional structure coming from a local coordinate system.
These extra structural assumptions are sufficient to improve the bounds on the critical probability in \Cref{t:mainThm-classH} even for graphs of superexponential order, and in particular for $n$-regular graphs, they even strengthen slightly the bounds from \Cref{thm:BaBoMohypercube} on the order of the third term in the expansion of the critical probability.

\begin{theorem}\label{t:mainThm-newH}
  Let $K \in \mathbb{N}$, let $(G_n)_{n \in \mathbb{N}}$ be a sequence of graphs in $\mathcal{H}(K)$ such that $\delta(G_n) \to \infty$ as $n \to \infty$ and consider the random majority bootstrap percolation process.
  Let $f\colon\mathbb{N} \times \mathbb{R} \to \mathbb{R}$ be a function defined as
  \[
    f(\ell, \lambda) \coloneqq \frac{1}{2}-\frac{1}{2} \sqrt{\frac{\log \ell}{\ell}}+ \lambda \frac{ \log \log \ell}{\sqrt{\ell \log \ell}}.
  \]
  Then
  \[
    \lim_{n \to \infty} \Phi(p, G_n) = \left\{
      \begin{array} {c@{\quad \textup{if} \quad}l}
        0 & p < f(\delta(G_n), \lambda) \quad \text{for any } \lambda < -\frac{1}{4}, \\[+1ex]
        1 & p > f(\Delta(G_n), \lambda) \quad \text{for any } \lambda> \frac{1}{2}.
      \end{array}\right.
  \]
\end{theorem}

The class $\mathcal{H}$ includes many of the natural families of high-dimensional graphs which lie in $\mathcal{F}$, in particular the hypercube and all Cartesian product graphs with base graphs of bounded size.
In particular, the bound $\lambda < -\frac{1}{4}$ in \Cref{t:mainThm-newH} improves the bound of $\lambda \leq -2$ in \Cref{thm:BaBoMohypercube} to this natural bottleneck.

Furthermore, $\mathcal{H}$ also includes a particular family of high-dimensional graphs of superexponential order arising naturally in the study of polytopes as well as Cayley graphs, the \emph{permutahedron}.

The vertex set of the permutahedron consists of the permutations of $[n+1]$ and two vertices are adjacent if the permutations differ only in a transposition of consecutive elements.
Thus, it can be viewed as the Cayley graph of the symmetric group $S_n$ generated by adjacent transpositions.
Equivalently, the permutahedron can also be viewed as the $1$-skeleton of the polytope which is the convex hull of the points in $\mathbb{R}^{n+1}$ given by the permutations, which can be seen to be a zonotope.
It is also the covering graph of the weak Bruhat lattice.
Even though the permutahedron arises naturally in many areas, it has so far mostly been studied from an algebraic and enumerative perspective and the analysis from a graph-theoretic viewpoint is much less developed \cite{CDE24}.

Like the hypercube, the permutahedron is highly-symmetric, even vertex-transitive, regular and has a very explicit description of its local structure.
But unlike the hypercube it is of superexponential order, which can cause challenges in its analysis.
As a corollary of \Cref{t:mainThm-newH} we obtain the following result for the permutahedron.

\begin{theorem}\label{t:permutahedron}
  Let $(P_n)_{n \in \mathbb{N}}$ be the sequence of $n$-dimensional permutahedra and consider the random majority bootstrap percolation process.
  Let
  \[
    p = \frac{1}{2}-\frac{1}{2} \sqrt{\frac{\log n}{n}}+ \lambda \frac{\log \log n}{\sqrt{n \log n}}.
  \]
  Then
  \[
    \lim_{n \to \infty} \Phi(p, P_n) = \left\{
      \begin{array} {c@{\quad \textup{if} \quad}l}
        0 & \lambda < -\frac{1}{4}, \\[+1ex]
        1 & \lambda > \frac{1}{2}.
      \end{array}\right.
  \]
\end{theorem}
\smallskip

Finally, we provide a first example of an irregular graph where we can improve on the bounds given by \Cref{t:mainThm-newH} and determine the second term in the expansion of the critical probability.
Explicitly, given some $q \in \mathbb{N}$ we consider the Cartesian product of $n$ copies of the star $K_{1, q}$ with $q$ leaves, given by $G_n \coloneqq \car_{i=1}^n K_{1, q}$.
This graph is very far from being regular, indeed the minimum degree of $G_n$ is $\delta(G_n)=n$, while the maximum degree is $\Delta(G_n)=qn$, and it is easy to see that almost all of its vertices have degrees very close to the average degree of $\frac{2q}{q+1}n$.
These irregular product graphs are known to demonstrate pathological behaviour in terms of their phase transitions under bond percolation, see \cite{CDE24,diskin2022Irregular}, and it is not clear if we should expect the critical probability for majority bootstrap percolation on such graphs to be controlled by the minimum, average, or maximum degree.

Whilst the presence of very high degree vertices heuristically seems to make it \emph{harder} to percolate in majority bootstrap percolation, we in fact show that, perhaps surprisingly, the location of the critical window for this graph sequence is controlled by the minimum degree, and so in particular is independent of the number of leaves $q$.

\begin{theorem}\label{t:star}
  Let $q\in \mathbb N$, let $G_n = \car_{i=1}^n K_{1, q}$ and consider the random majority bootstrap percolation process.
  Let
  \[
    p = \frac{1}{2}-\frac{1}{2} \sqrt{\frac{\log n}{n}}+ \lambda \frac{\log \log n}{\sqrt{n \log n}}.
  \]
  Then
  \[
    \lim_{n \to \infty} \Phi(p, G_n) = \left\{
      \begin{array} {c@{\quad \textup{if} \quad}l}
        0 & \lambda < -\frac{1}{4}, \\[+1ex]
        1 & \lambda > \frac{1}{2}.
      \end{array}\right.
  \]
\end{theorem}

\subsection{Key proof techniques}

In order to prove \Cref{t:mainThm-newH} we use a similar strategy as in \cite{BaBoMo2009,CEGK24}.
While the $1$-statement can directly be obtained from the proofs in there, the challenge lies in the $0$-statement.
In particular, if the graph is of superexponential order, then the union bound in the end of the proof fails.
We overcome this problem by a more careful analysis of the $Boot_k(\gamma)$ process introduced in \cite{BaBoMo2009} and also used in \cite{ CEGK24}.
This process dominates the original majority bootstrap percolation process.
We show that it stabilises after the first few steps and not the whole graph is infected at that point.

The key statement is a bound on the probability that a vertex $x$ is newly infected in the $(k+1)$-st step of the $Boot_k(\gamma)$-process.
This bound needs to be strengthened in comparison to the corresponding statement in \cite[Lemma 5.1]{CEGK24}.
This is achieved by parametrising the proof in terms of the number of cherries between the sphere at distance $k$ and $k+1$ from $x$.
The more cherries there are the worse the probability bound in the proof becomes, but on the other hand there are fewer choices of certain sets that require a union bound later.
Counting the cherries efficiently requires additional structural information that is present in the permutahedron and other high-dimensional graphs which are collected in the class $\mathcal{H}$.
This counting argument together with a careful calculation taking into account lower order terms allows to recover the second order term of the critical probability and strengthens the probability bounds such that a union over larger graphs is possible.

Note that in \cite{BaBoMo2009} the authors use a counting argument that considers hypergraphs and that the cherries considered here are also present in their work implicitly.
Nonetheless, the explicit counting presented here even allows us to improve the bound on $\lambda$ in \Cref{t:mainThm-newH} ($\lambda <-1/4$) in comparison to that in earlier work ($\lambda \leq -2$).

\smallskip

In order to get a matching upper and lower bound for an irregular graph, we rely on the particular structure of this graph.
While the $0$-statement of \Cref{t:star} is already implied by \Cref{t:mainThm-newH} the challenge lies in the $1$-statement.
For this purpose we choose a specific layer representation of the Cartesian product of stars.
By \Cref{t:mainThm-newH} the lowest few layers get infected, and the remaining task is to show that this infection spreads to the rest of the graph.
This crucially uses the fact that high-degree vertices in the Cartesian product of stars have more than half of their neighbours of lower degree.

\subsection{Organisation of the paper}

In \Cref{s:preliminary} we introduce some notation and state the probabilistic tools which we use, as well as recalling the class $\mathcal{F}$ of high-dimensional graphs defined in \cite{CEGK24}.
In \Cref{sec:H'} we define formally the class $\mathcal{H}$, which are our primary objects of study, and prove \Cref{t:mainThm-newH}.
In \Cref{s:permuta} we demonstrate that the permutahedron is indeed contained in $\mathcal{H}$.
Then we turn to the analysis of a specific irregular product graph, namely the Cartesian product of stars and prove \Cref{t:star} in \Cref{s:productstars}.
We close with a discussion of the limits of the methods used and some future directions in \Cref{s:discussion}.

\section{Preliminaries}\label{s:preliminary}

Given a graph $G=(V(G), E(G))$ we will write $\delta(G)$ for the minimum degree of $G$ and $\Delta(G)$ for the maximum degree of $G$.
Given two vertices $x,y \in V(G)$ we will write $\dist_{G}(x,y)$ for the \emph{distance} between $x$ and $y$ in $G$, that is, the length of the shortest $x$-$y$ path in $G$.
Given $k \in \mathbb{N} \cup \{0\}$ and $x \in V(G)$ we let
\[
  B_G(x,k)\coloneqq \{ y \in V(G) \colon \dist_G(x,y) \leq k \}
\]
be the \emph{ball of radius $k$} centred at $x$ and
\[
  S_G(x,k) \coloneqq B_G(x,k) \setminus B_G(x,k-1)
\]
be the \emph{sphere of radius $k$} centred at $x$, where $S_G(x, 0) \coloneqq \{x\}$.
Note in particular that $B_G(x, 0) = \{x\}$, $S_G(x, 1) = N_G(x)$, and $B_G(x, 1) = \{x\} \cup N_G(x)$.
When the underlying graph $G$ is clear from the context, we will omit the subscript in this notation.

We denote the set $\{1, \ldots, n\}$ by $[n]$.
Given a set $X$, the set $\binom{X}{i}$ denotes the set of all subsets of $X$ containing exactly $i$ elements.
Given $y,z \in \mathbb{R}$ we will write $ y \pm z$ to denote the interval $[y-z,y+z]$.
Similar to $O(\cdot)$ notation, an inclusion is meant whenever $\pm$ appears in the context of an equality.
For example, $x = y \pm z$ means that $x \in [y-z, y+z]$.

\subsection{Probabilistic tools}

In this section we state a number of probabilistic tools which are used throughout the paper.
We will assume knowledge of Markov's and Chebyshev's inequalities as standard (see~\cite{Alon2016Book} for basic probabilistic background).
The first statement is a standard form of the Chernoff bounds, see for example~\cite[Appendix A]{Alon2016Book}.
\begin{lemma}\label{l:Chernoff}
  Let $d \in \mathbb{N}$, $0 < p < 1$, and $X \sim \Bin(d,p)$.
  Then
  \begin{enumerate}[$(a)$]
  \item For every $t \geq 0$,
    \[
      \Pr[\big]{|X - dp| \geq t} \; \le \; 2\exp\left( -\frac{2t^2}{d} \right);
    \]
  \item\label{i:bigtail} For every $b \geq 1$,
    \[
      \Pr{X \geq bdp} \leq \left(\frac{e}{b}\right)^{bdp}.
    \]
  \end{enumerate}
\end{lemma}
However, we will need more precise tail bounds, mainly for the binomial distribution.
Using the central limit theorem, we can derive these from tail bounds for the standard normal distribution $\mathcal{N}(0, 1)$.
Proofs can be found in \cite[Section 5.6]{Asymptopia}.

\begin{lemma}\label{l:CLT}
  For $d \in \mathbb{N}$, $p=p(d)\in (0,1)$ and $f(d) = o(d^{1/6})$, it holds that
  \[
    \Pr*{\frac{\Bin(d, p)-d p}{\sqrt{dp(1-p)}} \geq f(d)} = (1+o(1)) \Pr*{\mathcal{N}(0, 1) \geq f(d)}.
  \]
  Moreover, if $f(d) \to \infty$ as $d \to \infty$, then the probability that the standard normal distribution exceeds $f(d)$ satisfies
  \[
    \Pr{\mathcal{N}(0,1) \geq f(d)} = \frac{1+o(1)}{f(d) \sqrt{2 \pi}} \exp\left(-\frac{f(d)^2}{2}\right).
  \]
\end{lemma}

Finally, since binomial random variables can be written as sums of independent Bernoulli random variables, Hoeffding's inequality~\cite{Hoeffding1963} readily implies the following lemma.
\begin{lemma}\label{l:sumBin}
  Let $k,d_1,\ldots,d_k \in \mathbb{N}$ and $p \in (0,1)$.
  Let $X_i \sim \Bin(d_i,p)$ for each $i \in [k]$, let $Y = \sum_{i=1}^k iX_i$, and let $D(k) = \sum_{i=1}^k i^2 d_i$.
  Then, for every $\tau > 0$,
  \[
    \Pr[\big]{Y \ge \Ex{Y} + \tau} \le \exp\left( - \frac{2\tau^2}{D(k)} \right) \le \exp\left( - \frac{2p \tau^2}{k\cdot\Ex{Y}} \right).
  \]
\end{lemma}

\subsection{The class $\mathcal{F}$ of high-dimensional geometric graphs from \cite{CEGK24}}\label{sec:definitionH}

We recall the definition of the class $\mathcal{F}=\cup_{K\in \mathbb{N}} \mathcal{F}(K)$ of high-dimensional geometric graphs introduced in \cite{CEGK24}.
Roughly, one can think of these graphs as having a structure which is in some sense \emph{close} to being governed by some \emph{local coordinate system}, where edges are only allowed between vertices which differ in a single coordinate (see \cite[Section 3]{CEGK24} for more details).

Given $K \in \mathbb{N}$, the class $\mathcal{F}(K)$ is defined recursively by taking the single-vertex graph $K_1 \in \mathcal{F}(K)$ and then taking every graph $G$ satisfying \Cref{c:regular,c:backwards,c:localconnection,c:projection,c:partitioncond,c:bound} below.

\begin{property}[Locally almost regular]\makeatletter\edef\@currentlabel{(\theproperty)}\makeatother\label{c:regular}
  For every $x \in V(G)$, $\ell \in \mathbb{N}$ and $y \in S(x, \ell)$, $$|d(x)-d(y)| \leq K \ell.$$
\end{property}

\begin{property}[Bounded backwards expansion]\makeatletter\edef\@currentlabel{(\theproperty)}\makeatother\label{c:backwards}
  For every $x \in V(G)$, $\ell \in \mathbb{N}$ and $y \in S(x, \ell)$,
  \[
    |N(y) \cap B(x, \ell)| \leq K\ell.
  \]
\end{property}

\begin{property}[Typical local structure]\makeatletter\edef\@currentlabel{(\theproperty)}\makeatother\label{c:localconnection}
  For every $x\in V(G)$ there is a set $D \subseteq V(G) \setminus \{x\}$ of \emph{non-typical} vertices such that for every $\ell \in \mathbb{N}$ the following hold:
  \begin{enumerate}[$(i)$]
  \item\label{i:small} $|D\cap S(x,\ell)| \leq  K^{\ell-1} (d(x))^{\ell-1}$;
  \item\label{i:sparse} $|D\cap N(y)| \leq K\ell$ for every vertex $y \in S(x, \ell) \setminus D$;
  \item\label{i:cherry} every two vertices in $S(x, \ell) \setminus D$ have at most one common neighbour in $S(x, \ell+1) \setminus D$.
  \end{enumerate}
\end{property}

If a graph satisfies \Cref{c:localconnection} we denote the set of \emph{typical} vertices at distance $\ell$ from $x$ by $S_0(x, \ell)$, i.e., for $\ell \in \mathbb{N}$ we set
\[
  S_0(x, \ell) \coloneqq S(x, \ell) \setminus D.
\]

\begin{property}[Projection]\makeatletter\edef\@currentlabel{(\theproperty)}\makeatother\label{c:projection}
  For every $x \in V(G)$, $\ell \in \mathbb{N}$ and $y \in S(x, \ell)$, there is a subgraph $G(y)$ of $G$ such that the following hold:
  \begin{enumerate}[$(i)$]
  \item $y \in V(G(y))$;
  \item $G(y) \in \mathcal{F}(K)$;
  \item $V(G(y)) \cap B(x, \ell-1) = \emptyset$; \item $ |d_{G(y)}(w) - d_{G}(w)| \leq K \ell$ for all $w \in V(G(y))$.
  \end{enumerate}
\end{property}

\begin{property}[Separation]\makeatletter\edef\@currentlabel{(\theproperty)}\makeatother\label{c:partitioncond}
  For every $x \in V(G)$, $\ell \in \mathbb{N}$ and $y \in S_0(x, \ell)$, we have
  \[
    |B(y, 2\ell-1) \cap S_0(x, \ell)| \leq \ell K^{\ell-1}d(x)^{\ell-1}.
  \]
\end{property}

\begin{property}[Exponential order]\makeatletter\edef\@currentlabel{(\theproperty)}\makeatother\label{c:bound}
  $G$ has at most $\exp(K \delta(G))$ vertices.
\end{property}

\section{High-dimensional graphs of higher order --- \Cref{t:mainThm-newH}}\label{sec:H'}

In this section we consider a slightly different class $\mathcal{H}$ of high-dimensional graphs.
By strengthening \Cref{c:localconnection}, we will be able to give sharper bounds on the probability that a vertex is infected after $i$ rounds (\Cref{l:xinAi+1-Ai}) which will allow us not only to weaken \Cref{c:bound}, but also to give the more refined bounds on the critical probability in \Cref{t:mainThm-newH}.

More explicitly, we consider the following strengthening of  \Cref{c:localconnection}.

\let\oldproperty\theproperty\renewcommand{\theproperty}{P3$'$}
\begin{property}[Stronger typical local structure]\makeatletter\edef\@currentlabel{(\theproperty)}\makeatother\label{c:localconnectionstrong}
  For every $x\in V(G)$ there is a set $D \subseteq V(G) \setminus \{x\}$ of \emph{non-typical} vertices such that for every $\ell \in \mathbb{N}$ the following hold:
  \begin{enumerate}[$(i)$]
  \item\label{i:smallnew} $|D\cap S(x,\ell)| \leq  K^{\ell-1} (d(x))^{\ell-1}$;
  \item\label{i:sparsenew} $|D\cap N(y)| \leq K\ell$ for every vertex $y \in S(x, \ell) \setminus D$;
  \item\label{i:cherrynew} every two vertices in $S(x, \ell) \setminus D$ have at most one common neighbour in $S(x, \ell+1) \setminus D$;
  \item\label{i:backwardstrong} every vertex $v \in S(x, \ell) \setminus D$ has at most $\ell$ neighbours in $S(x, \ell-1)$;
  \item\label{i:cherries} for every cherry $uvw$ with $u, w \in S(x, \ell) \setminus D$ and $v \in S(x, \ell+1) \setminus D$, it holds that $(S(x, \ell-1) \setminus D) \cap N(u) \cap N(w) \neq \varnothing$.
  \end{enumerate}
\end{property}
\let\theproperty\oldproperty\addtocounter{property}{-1}

\begin{figure}
  \centering
  \begin{tikzpicture}[scale=1, thick, every node/.style={circle}]
    \node (V) at (0,0) [circle,draw, fill, scale=0.6]  {};
    \begin{scope}
      \draw[clip, rotate=90] (0, -3) ellipse (2.5cm and 0.6cm);
      \fill[gray] (2.6,2) -- (2.6,4) -- (3.4,4) -- (3.4,2) -- cycle;
    \end{scope}
    \begin{scope}
      \draw[clip, rotate=90] (0, -4.7) ellipse (3cm and 0.7cm);
      \fill[gray] (4.25,2.3) -- (4.25,5) -- (5.1,5) -- (5.1,2.3) -- cycle;
    \end{scope}

    \node[] at (0,0.4) {$x$};
    \draw[rotate=90] (0, -1.5) ellipse (2cm and 0.5cm);
    \node[] at (1.5,2.4) {$N(x)$};
    \draw[rotate=90] (0, -3) ellipse (2.5cm and 0.6cm);
    \node[] at (3,2.9) {$S(x, 2)$};
    \draw[rotate=90] (0, -4.7) ellipse (3cm and 0.7cm);
    \node[] at (4.7,3.4) {$S(x, 3)$};

    \node (W) at (4.8,0) [circle,draw, fill, scale=0.6] {};
    \node[] at (4.8,0.4) {$v$};
    \node (W2) at (3,-0.9) [circle,draw, fill, scale=0.6] {};
    \node[] at (3,-1.3) {$w$};
    \path[draw=black] (4.8,0) -- (3,-0.9);

    \node (U) at (3,1) [circle,draw, fill, scale=0.6] {};
    \node[] at (3,1.4) {$u$};
    \node (U2) at (1.5,0) [circle,draw, fill=blue, scale=0.6] {};
    \node[] at (1.5,0.4) {$z$};

    \path[draw=blue] (3,1) -- (1.5,0);
    \path[draw=black] (4.8,0) -- (3,1);
    \path[draw=blue] (3,-0.9) -- (1.5,0);
  \end{tikzpicture}
  \caption{The additional property in \Cref{c:localconnectionstrong}\ref{i:cherries} for $\ell=2$ asks that for any cherry $uvw$ there is a common neighbour $z$ of $u$ and $w$ in $N(x)$.}\label{f:cherry}
\end{figure}
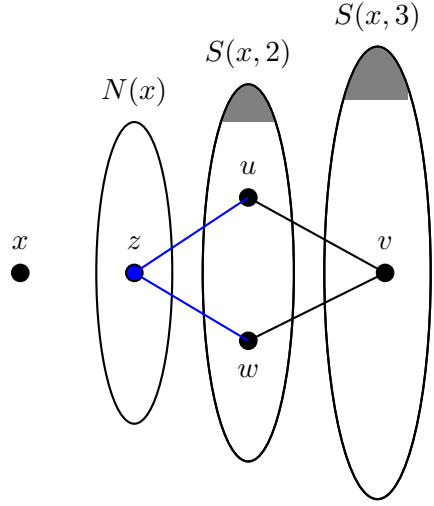

Note that \Cref{c:localconnectionstrong} requires two additional properties \ref{i:backwardstrong} and \ref{i:cherries}, in comparison to \Cref{c:localconnection} (see \Cref{f:cherry} for a pictorial description of property \ref{i:cherries}).
As with \Cref{c:localconnection}, if a graph satisfies \Cref{c:localconnectionstrong} we denote the set of \emph{typical} vertices at distance $\ell$ from $x$ by $S_0(x, \ell)$, i.e., for $\ell \in \mathbb{N}$ we set
\[
  S_0(x, \ell) \coloneqq S(x, \ell) \setminus D.
\]

Under the assumption \Cref{c:localconnectionstrong}, we can relax \Cref{c:bound} as follows.
\let\oldproperty\theproperty\renewcommand{\theproperty}{P6$'$}
\begin{property}[Larger bounded order]\makeatletter\edef\@currentlabel{(\theproperty)}\makeatother\label{c:bound-larger}
  $G$ has at most $\exp\left(\frac{\delta(G)^{3/2} \log \log \delta(G)}{\log^2 \delta(G)}\right)$ vertices.
\end{property}

Note that with the methods presented, the exponent of $3/2$ is best-possible, see \Cref{rem:optimizing}.

We are now ready to define the class $\mathcal{H}=\cup_{K\in \mathbb{N}} \mathcal{H}(K)$.
Given $K \in \mathbb{N}$ the class $\mathcal{H}(K)$ of graphs is defined by taking the single-vertex graph $K_1 \in \mathcal{H}(K)$ as well as every graph $G$ satisfying properties \ref{c:regular}, \ref{c:backwards}, \ref{c:localconnectionstrong}, \ref{c:projection} (with $G(y) \in \mathcal{H}(K)$ instead of $G(y) \in \mathcal{F}(K)$), \ref{c:partitioncond} and \ref{c:bound-larger}.

As with $\mathcal{F}$, one can think of the graphs in $\mathcal{H}$ as having a structure that is close to be governed by some \emph{local coordinate system}.
Let us motivate the extra conditions in \ref{c:localconnectionstrong} from this viewpoint.

Heuristically a \emph{typical} vertex $w \in S_0(x,\ell)$ differs from $x$ in precisely $\ell$ coordinates, and so there should be at most $\ell$ neighbours of $w$ in $S_0(x,\ell-1)$.
On the other hand, two typical vertices in $S_0(x,\ell)$ should have a common neighbour in $S_0(x,\ell+1)$ precisely when they differ from each other in a pair of coordinates in which they both differ from $x$, in which case they should also have a common neighbour in $S_0(x,\ell-1)$.

It is easy to adapt the proof of \cite[Lemma 8.2]{CEGK24} to show that every Cartesian product graph whose base graphs have bounded size lie in $\mathcal{H}$ (\Cref{l:prodinH}), and so in particular, the hypercube, lies in $\mathcal{H}$.
Furthermore, we will later show that the permutahedron is contained in $\mathcal{H}$ (see \Cref{s:permuta}) and thus \Cref{t:mainThm-newH} implies \Cref{t:permutahedron}.

We note, however, that \Cref{c:localconnectionstrong}\ref{i:cherries} does not hold for all graphs in the class $\mathcal{F}$ considered in \cite{CEGK24}.
In particular, for the middle layer graph \Cref{c:localconnectionstrong}\ref{i:cherries} is vacuous for $\ell=1$, and is not true for $\ell \geq 2$, since the graph has girth $6$.

\subsection{The $1$-statement}\label{s:proofof1statement}

For the $1$-statement in \Cref{t:mainThm-newH}, we show that for any $\lambda > \frac{1}{2}$ and
\[
  p \geq \frac{1}{2} - \frac{1}{2} \sqrt{\frac{\log \Delta}{\Delta}} + \lambda \frac{\log \log \Delta}{\sqrt{\Delta \log \Delta}}
\]
the random majority bootstrap percolation process on a graph $G \in \mathcal{H}(K)$ percolates after eleven rounds whp, i.e.,
\[
  \Pr{\exists x \in V(G): x \notin A_{11}} =o(1).
\]
In fact, this essentially already follows from the work in \cite{CEGK24}, see Remark 6.5 there.
\begin{lemma}[Lemma 4.1 and Remark 6.5 in \cite{CEGK24}, informal]\label{l:constant}
  Let $G$ be a graph that satisfies \Cref{c:regular,c:backwards,c:localconnection} for some $K \in \mathbb{N}$.
  Let $x\in V(G)$, $\lambda > \frac{1}{2}$ and
  \[
    p \geq \frac{1}{2} - \frac{1}{2} \sqrt{\frac{\log d(x)}{d(x)}} + \lambda \frac{\log \log d(x)}{\sqrt{d(x) \log d(x)}}.
  \]
  Then
  \[
    \Pr*{x \in A_2} \geq \frac{3}{4} + o(1).
  \]
\end{lemma}

A proof of \Cref{l:constant} can be found in \cite[Appendix A]{CEGK24}.
With this statement at hand, one can follow the proof in the main body of \cite{CEGK24} to show that the probability that a vertex is not infected after $i$ rounds shrinks very quickly.
From a constant probability after two rounds it drops to an exponentially small probability after five rounds and a doubly-exponentially small probability after eleven rounds.

\begin{lemma}[Lemmas 4.2 and 4.3 in \cite{CEGK24}]\label{l:super-exponential}
  Let $G$ be a graph that satisfies \Cref{c:regular,c:backwards,c:localconnection,c:projection,c:partitioncond} for some $K \in \mathbb{N}$.
  Let $x\in V(G)$, $\lambda > \frac{1}{2}$ and
  \[
    p \geq \frac{1}{2} - \frac{1}{2} \sqrt{\frac{\log d(x)}{d(x)}} + \lambda \frac{\log \log d(x)}{\sqrt{d(x) \log d(x)}}.
  \]
  Then there is a $\beta_1>0$ (independent of $x$) such that
  \[
    \Pr[\big]{x \notin A_{5}}< \exp\left(-\beta_1 d(x)\right),
  \]
  and furthermore there is a $\beta_2>0$ (independent of $x$) such that
  \[
    \Pr[\big]{x \notin A_{11}}< \exp\left(-\beta_2 d(x)^2\right).
  \]
\end{lemma}

Since, by \Cref{c:bound-larger}, any graph in $G \in \mathcal{H}(K)$ contains $\exp\left( o\left( \delta(G)^2 \right) \right)$ vertices, the $1$-statement of \Cref{t:mainThm-newH} follows immediately from \Cref{l:super-exponential} via a union bound argument.

\subsection{The $0$-statement}\label{s:proofof0statement}

For the $0$-statement of \Cref{t:mainThm-newH}, we show that for any $\lambda < -1/4$ and
\[
  p < \frac{1}{2}-\frac{1}{2} \sqrt{\frac{\log \dmin}{\dmin}} - \lambda \frac{\log \log \dmin}{\sqrt{\dmin \log \dmin}},
\]
the random majority bootstrap process on $G \in \mathcal{H}(K)$ whp does not percolate, i.e., $$ \Phi(p,G)\coloneqq \Pr{\textbf{A}_p \text{ percolates on } G} \to 0.$$

Instead of directly analysing the majority bootstrap percolation process, we will analyse a generalised process which dominates the original majority bootstrap percolation process.
The \emph{$\Boot_k(\gamma)$ process} was introduced in~\cite{BaBoMo2009} and we recall its definition: Given $k \in \mathbb{N}$ and some function $\gamma \colon V(G) \to \mathbb{R}^+$, we recursively define $\hat{A}_0 \coloneqq A_0$ and, for each $\ell \in \mathbb{N} \cup \{0\}$,
\[
  \hat{A}_{\ell+1} \coloneqq \hat{A}_\ell \cup \left\{ x \in V(G) : \left|N(x) \cap \hat{A}_\ell \right| \geq \frac{d(x)}{2} - \max\{0, k - \ell\}\cdot \gamma(x) \right\}.
\]
In other words, the initial infection set is the same as in the majority bootstrap percolation process, but the infection spreads more easily in the first $k$ rounds.
More precisely, a vertex $x$ is infected in the first round if it has $d(x)/2 - k\cdot\gamma(x)$ infected neighbours, and this requirement is gradually strengthened over the first $k$ rounds.
After the $k$th round, the process evolves exactly as the majority bootstrap percolation process would do.

It is clear that this process dominates the majority bootstrap percolation process in that, for any fixed $A_0$, we have that $A_i\subseteq \hat{A}_i$ for all $i \in \mathbb{N}$.
Crucially, however, if a vertex $x$ becomes infected in round $\ell+1 \leq k$ of the $\Boot_k(\gamma)$-process, then at least $\gamma(x)$ of its neighbours must have become infected in round $\ell$.
This simplifies the task of showing a vertex does \emph{not} become infected.

For our application, we fix $k=3$ and
\[
  \gamma(x)= \sqrt{\frac{d(x)}{\log d(x)}},
\]
which matches the choice in \cite{BaBoMo2009}.
In \cite{CEGK24} the authors of this work presented a simpler argument, using a slightly larger value of $\gamma(x)$ and with $k=2$, which allowed them to avoid a technical counting argument from \cite{BaBoMo2009}, at the cost of slightly weakening the bounds on the critical probability.
Here, to recover and strengthen the bounds from \cite{BaBoMo2009}, we have to be more careful.

For their analysis, Balogh, Bollob\'{a}s and Morris developed a counting lemma (\cite[Lemma 5.4]{BaBoMo2009}) for a hypergraph that encodes certain cherries in the hypercube.
In order to improve their bounds on $\lambda$, we instead do the cherry counting directly and in broader generality.

The main step towards the proof of the $0$-statement is to show that in the $\Boot_k(\gamma)$-process the probability that a vertex becomes infected after the first step is shrinking quickly.
We will be able to leverage this information to show that whp no new vertices are infected after the third round and to deduce that whp the process does not percolate.
\begin{lemma}\label{l:xinAi+1-Ai}
  Let $i \in \{2, 3\}$, $x \in V(G)$, $d=d(x)$, $\gamma=\gamma(x)$, $\lambda < -1/4$ and $p = \frac{1}{2}-\frac{1}{2} \sqrt{\frac{\log d}{d}} + \lambda \frac{\log \log d}{\sqrt{d \log d}}$.
  If there is a $K \in \mathbb{N}$ such that $G \in \mathcal{H}(K)$, then there is a $\beta >0$ (independent of $x$) such that
  \[
    \Pr{x \in \hat{A}_{i+1} \setminus \hat{A}_i} \leq \exp\left(-\beta \gamma^i \log \log d\right).
  \]
\end{lemma}

In order to prove \Cref{l:xinAi+1-Ai} we will introduce the notion of a \emph{witness}, a structure witnessing that a vertex $x$ is infected in the second or third round, and prove several claims that lead to the proof of \Cref{l:xinAi+1-Ai} in \Cref{proof_xinAi+1-Ai}.
Below we let $x \in V(G)$ be given and write $d=d(x), \gamma=\gamma(x)$.

First, observe that if $w \in V(G)$ is such that $\dist(x, w)$ is constant, then by \Cref{c:regular} and the asymptotic estimates $\log (d + O(1))/\log d = 1 + O(1/d)$ and $\sqrt{1 + O(1/d)} = 1 + O(1/d)$,
\begin{equation}\label{e:gammaconstant}
  \gamma(w) = \gamma \cdot \left(1 + O\left(\frac{1}{d}\right)\right) = \gamma + o(1).
\end{equation}

In the following, fix $1 \leq i \leq 3$.
In order to prove that the vertex $x$ has low probability to become infected in step $i+1$, we define a \emph{witness} to be a tuple of $i$ sets $(W_j)_{1 \leq j \leq i}$ such that
\begin{equation}\label{e:witness-def}
  W_j \subseteq S_0(x, j), \qquad W_{j+1} \subseteq N(W_j) \qquad \text{ and } \qquad |W_j|= \frac{(\gamma - 7K)^j}{j!}.
\end{equation}
The \emph{weight} of a witness $(W_j)_{1 \leq j \leq i}$ is defined as
\[
  \zeta(W_i) \coloneqq e(W_i, S_0(x, i+1)).
\]
We observe that, for every witness $(W_j)_{1 \leq j \leq i}$, we have
\begin{equation}\label{e:witnessobs}
  \zeta(W_i) = |W_i|(d\pm 3iK),
\end{equation}
since every vertex of $W_i$ has degree $d\pm iK$ by \Cref{c:regular} and every vertex of $S_0(x,i)$ has at most $2iK$ neighbours outside $S_0(x,i+1)$ by Properties \ref{c:backwards} and \ref{c:localconnectionstrong}\ref{i:sparsenew}.
For an illustration of a witness see \Cref{f:witnessfori3}.

Furthermore, given the random set $\hat{A}_0$ we consider the random variable
\[
  Z(W_i) \coloneqq e(W_i, S_0(x, i+1) \cap \hat{A}_0).
\]
The definition of a witness is motivated by the following claim.

\begin{claim}\label{c:existencewitness}
  Let $i \in \{1,2,3\}$.
  If $x \in \hat{A}^{(i+1)} \setminus \hat{A}^{(i)}$, then there exists a witness $(W_j)_{1 \leq j \leq i}$ with $s=|W_i|$ such that
  \begin{equation}\label{e:0-statement-dev}
    Z(W_i) \geq \Ex{Z(W_i)} + s\left(\frac{1}{2} \sqrt{d \log d} - \lambda \log \log d\sqrt{\frac{d}{\log d}}-3 \gamma - 9iK\right).
  \end{equation}
\end{claim}
\begin{proof}
  We will inductively construct, for $1 \leq j \leq i$, sets $W_j\subseteq S_0(x, j) \cap \left(\hat{A}_{i-j+1} \setminus \hat{A}_{i-j}\right)$ satisfying~\eqref{e:witness-def}.
  By definition of the $Boot_3(\gamma)$-process, entering $\hat{A}_{i+1}$ requires $x$ to have $d/2 - (3-i)\gamma$ infected neighbours, while entering $\hat{A}_i$ requires only $d/2 - (3-i+1)\gamma$ ones.
  Therefore, the event $x \in \hat{A}_{i+1} \setminus \hat{A}_i$ implies the existence of a set $W_1' \subseteq S(x, 1) \cap \left(\hat{A}_i \setminus \hat{A}_{i-1}\right)$ of recently infected neighbours of $x$ of size $|W_1'|=\gamma$.
  Since $|D \cap S(x, 1)|\leq 1$ by \Cref{c:localconnectionstrong}\ref{i:smallnew}, we may take a subset $W_1 \subseteq W_1' \cap S_0(x,1)$ of size $\gamma - 7K$.
  This settles the induction start.

  Now suppose we have a set $W_j \subseteq \left(S_0(x, j) \cap \left(\hat{A}_{i-j+1} \setminus \hat{A}_{i-j}\right)\right)$ of size $\frac{(\gamma - 7K)^j}{j!}$ for some $1 \leq j < i$.
  Again, by definition of the $Boot_3(\gamma)$-process, each $w \in W_j$ has at least $\gamma(w) \geq \gamma - 1$ (see \eqref{e:gammaconstant}) neighbours in $\hat{A}_{i-j} \setminus \hat{A}_{i-j-1}$.
  By \Cref{c:backwards} each such $w$ has at most $jK$ neighbours in $B(x, j)$.
  Furthermore, by \Cref{c:localconnectionstrong}\ref{i:sparsenew} each such $w$ has at most $jK$ neighbours in $D$.
  Thus, for each $w \in W_j$ there exists a set $W_{j+1, w} \subseteq \left(N(w) \cap S_0(x, j+1) \cap \left(\hat{A}_{i-j} \setminus \hat{A}_{i-j-1}\right)\right)$ of size at least $\gamma - 2jK -1 \geq \gamma - 7K$.
  Moreover, every element of
  \[
    W_{j+1}' \coloneqq \bigcup_{w \in W_j} W_{j+1, w}
  \]
  has at most $j+1$ neighbours in $W_j$ by \Cref{c:localconnectionstrong}\ref{i:backwardstrong} and therefore
  \[
    |W_{j+1}'| \geq |W_j| \frac{\gamma-7K}{j+1} = \frac{(\gamma-7K)^{j+1}}{(j+1)!}.
  \]
  Thus, we may pick a subset $W_{j+1} \subseteq W_{j+1}'$ of the right size.

  It remains to show that \eqref{e:0-statement-dev} holds.
  By definition of the $Boot_3(\gamma)$-process and~\Cref{c:regular}, each $w \in W_i$ has at least $d(w)/2-3 \gamma(w) \geq d/2 - 3\gamma - 4iK$ neighbours in $\hat{A}_0$.
  By \Cref{c:backwards,c:localconnectionstrong}\ref{i:sparsenew}, at least $d/2-3 \gamma -6iK$ of these neighbours lie in $S_0(x, i+1) \cap \hat{A}_0$.
  Setting
  \[
    \zeta \coloneqq \zeta(W_i)= e(W_i, S_0(x, i+1)) \qquad \text{and} \qquad Z \coloneqq Z(W_i) = e(W_i, S_0(x, i+1) \cap \hat{A}_0),
  \]
  we obtain $\Ex{Z}= \zeta p$.
  On the other hand, $W_i \subseteq \hat{A}_1 \setminus \hat{A}_0$ by construction.
  Hence,
  \begin{equation}\label{e:existencewitness:almost}
    Z \geq |W_i|\left(\frac{d}{2}-3 \gamma - 6iK\right) = |W_i|dp + |W_i|\left(\frac{1}{2} \sqrt{d \log d} - \lambda \log \log d \sqrt{\frac{d}{\log d}} - 3 \gamma - 6iK \right),
  \end{equation}
  where the equality uses the definition of $p$.
  By~\eqref{e:witnessobs}, it holds that $|W_i|d \geq \zeta - 3iK|W_i|$, and therefore $|W_i|dp \geq \Ex{Z} - 3iK|W_i|$.
  Together with~\eqref{e:existencewitness:almost}, this implies the desired claim.
\end{proof}

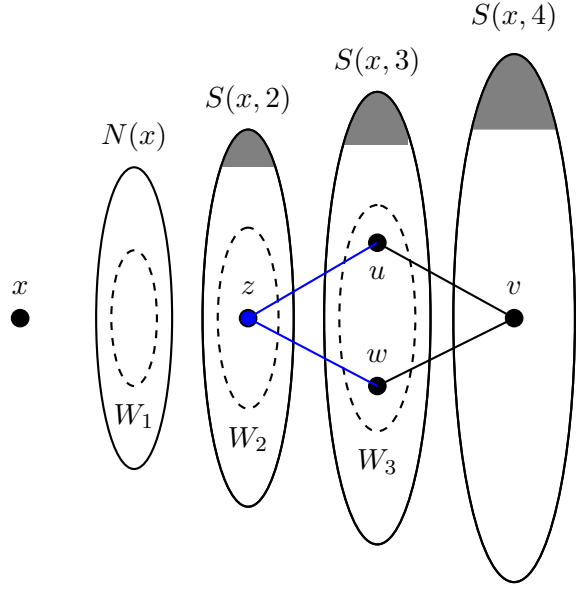
\begin{figure}
  \centering
  \begin{tikzpicture}[scale=1, thick, every node/.style={circle}]
    \node (V) at (0,0) [circle,draw, fill, scale=0.6]  {};
    \begin{scope}
      \draw[clip, rotate=90] (0, -3) ellipse (2.5cm and 0.6cm);
      \fill[gray] (2.6,2) -- (2.6,4) -- (3.4,4) -- (3.4,2) -- cycle;
    \end{scope}
    \begin{scope}
      \draw[clip, rotate=90] (0, -4.7) ellipse (3cm and 0.7cm);
      \fill[gray] (4.25,2.3) -- (4.25,5) -- (5.1,5) -- (5.1,2.3) -- cycle;
    \end{scope}
    \begin{scope}
      \draw[clip, rotate=90] (0, -6.5) ellipse (3.5cm and 0.8cm);
      \fill[gray] (5.7,2.5) -- (6,5) -- (6.9,5) -- (7.4,2.5) -- cycle;
    \end{scope}

    \node[] at (0,0.4) {$x$};
    \draw[rotate=90] (0, -1.5) ellipse (2cm and 0.5cm);
    \node[] at (1.5,2.4) {$N(x)$};
    \draw[rotate=90] (0, -3) ellipse (2.5cm and 0.6cm);
    \node[] at (3,2.9) {$S(x, 2)$};
    \draw[rotate=90] (0, -4.7) ellipse (3cm and 0.7cm);
    \node[] at (4.7,3.4) {$S(x, 3)$};
    \draw[rotate=90] (0, -6.5) ellipse (3.5cm and 0.8cm);
    \node[] at (6.5,4) {$S(x, 4)$};

    \draw[rotate=90,dashed] (0, -1.5) ellipse (0.9cm and 0.3cm);
    \node[] at (1.5,-1.3) {$W_1$};
    \draw[rotate=90,dashed] (0, -3) ellipse (1.2cm and 0.4cm);
    \node[] at (3,-1.6) {$W_2$};
    \draw[rotate=90,dashed] (0, -4.7) ellipse (1.5cm and 0.5cm);
    \node[] at (4.7,-1.9) {$W_3$};

    \node (W) at (6.5,0) [circle,draw, fill, scale=0.6] {};
    \node[] at (6.5,0.4) {$v$};
    \node (W2) at (4.7,-0.9) [circle,draw, fill, scale=0.6] {};
    \node[] at (4.7,-0.5) {$w$};
    \path[draw=black] (6.5,0) -- (4.7,-0.9);
    \node (U) at (4.7,1) [circle,draw, fill, scale=0.6] {};
    \node[] at (4.7,0.6) {$u$};
    \node (U2) at (3,0) [circle,draw, fill=blue, scale=0.6] {};
    \node[] at (3,0.4) {$z$};

    \path[draw=blue] (4.7,1) -- (3,0);
    \path[draw=black] (6.5,0) -- (4.7,1);
    \path[draw=blue] (4.7,-0.9) -- (3,0);
  \end{tikzpicture}
  \caption{A witness $(W_1, W_2, W_3)$ for $i=3$. We consider edges from $W_3$ to $S_0(x, 4)$, in particular cherries like $uvw$. By \Cref{c:localconnectionstrong}\ref{i:cherries} each such cherry can be completed to a $C_4$ with $z=z(uvw)$ from $S_0(x, 2)$.}\label{f:witnessfori3}
\end{figure}

To obtain a probability bound for the event that $Z(W_i)$ differs from its expectation, we consider \emph{cherries} between $W_i$ and the next layer $S_0(x, i+1)$.
More precisely, for a witness $(W_j)_{1 \leq j \leq i}$, let $m(W_i)$ denote the number of copies of $K_{1, 2}$ with centre in $S_0(x, i+1)$ and endpoints in $W_i$ (see \Cref{f:witnessfori3}).
As a warm-up, let us prove that every witness $(W_j)_{1 \leq j \leq i}$ with $|W_i|=s$ satisfies
\begin{equation}\label{eq:mWi}
  m(W_i)\leq \frac{isd}{2} + O(s).
\end{equation}
Indeed, if we fix one endpoint $w \in W_i$ of a cherry, then we can choose the midpoint $v \in S_0(x, i+1)$ of the cherry in at most $d(w) \leq d+iK$ ways and, by \Cref{c:localconnectionstrong}\ref{i:backwardstrong} there are at most $i$ choices for the other endpoint of the cherry, which must be another neighbour of $v \in S_0(x,i+1)$ in $W_i \subseteq S_0(x,i)$.
In this way, we count each cherry twice, establishing \eqref{eq:mWi}.

We bound the probability of \eqref{e:0-statement-dev} in terms of a function of $\alpha \coloneqq \frac{m(W_i)}{sd}$, where we note that $\alpha \geq 0$ is bounded above by a constant by the argument above.
\begin{claim}\label{c:probboundalpha}
  Let $i \in \{2, 3\}, \alpha \geq 0$ and $\lambda<-1/4$. For a witness $(W_j)_{1 \leq j \leq i}$ with $s = |W_i|$ and $m(W_i)=\alpha s d$, define the event
  \[
    \mathcal{E}(W_i) \coloneqq \left\{ Z(W_i) \geq \Ex{Z(W_i)} + s\left(\frac{1}{2} \sqrt{d \log d} - \lambda \log \log d\sqrt{\frac{d}{\log d}} - 3 \gamma - 9iK\right) \right\}.
  \]
  Then
  \begin{equation}\label{e:boundonprob}
    \Pr{\mathcal{E}(W_i)} \leq \exp\left( -\left(\frac{s}{2} - \frac{\alpha s}{1+2\alpha}\right)\left(\log d - (4+o(1))\lambda \log \log d\right)\right).
  \end{equation}
\end{claim}
\begin{proof}
  For $j \in [i+1]$, let $X_j$ be the set of elements of $S_0(x,i+1)$ having $j$ neighbours in $W_i$.
  By \Cref{c:localconnectionstrong}\ref{i:backwardstrong} every vertex in $S_0(x, i+1)$ sends at most $i+1$ edges to $W_i$ and so
  \[
    \zeta = e(W_i, S_0(x, i+1)) = \sum_{j=1}^{i+1} j \cdot |X_j|.
  \]
  On the other hand,
  \[
    m \coloneqq m(W_i) = \sum_{j=1}^{i+1} \binom{j}{2} |X_j| = \frac{1}{2} \left( \sum_{j=1}^{i+1} j^2 \cdot |X_j|\right) - \frac{\zeta}{2},
  \]
  and so $\sum_{j=1}^{i+1} j^2 \cdot |X_j| = 2m + \zeta$.
  Since $Z(W_i)$ is the sum of independent random variables with distribution $j \cdot \Bin(|X_j|,p)$ for $j \in [i+1]$, we can apply Hoeffding's inequality (\Cref{l:sumBin}), noting that $D(i+1)=2m + \zeta$, to conclude that for any $t \geq 0$ we have
  \begin{equation}\label{eq:deviationhoeffding}
    \Pr{Z(W_i) \geq \Ex{Z(W_i)} + t} \leq \exp \left( - \frac{2t^2}{2m+\zeta}\right).
  \end{equation}
  We will apply \eqref{eq:deviationhoeffding} with
  \[
    t =s\left(\frac{1}{2} \sqrt{d \log d} - \lambda \log \log d\sqrt{\frac{d}{\log d}} - 3 \gamma - 9iK\right) = \frac{s\sqrt{d}}{2} \left(\sqrt{\log d} - (2 + o(1))\frac{\lambda \log\log d}{\sqrt{\log d}}\right),
  \]
  noting that
  \begin{equation}\label{eq:boundont2}
    t^2 = \frac{s^2 d}{4}\bigl(\log d - (4 + o(1)) \lambda \log \log d\bigr).
  \end{equation}
  In particular, recalling that $\zeta = s(d + O(1))$ from \eqref{e:witnessobs} and that $\alpha = m/sd = O(1)$ from \eqref{eq:mWi}, for $d$ large enough \eqref{eq:boundont2} implies that
  \begin{align}\label{eq:terminhoeffding}
    \frac{2t^2}{2m + \zeta} &= \left(\frac{s}{4\alpha + 2 + O(1/d)}\right)\bigl(\log d - (4 + o(1)) \lambda \log \log d\bigr)\nonumber\\
                            &= \left(\frac{s}{2} - \frac{\alpha s}{1 + 2\alpha}\right)\bigl(\log d - (4+o(1))\lambda \log \log d\bigr).
  \end{align}
  Combining \eqref{eq:deviationhoeffding} and \eqref{eq:terminhoeffding} yields the claim.
\end{proof}

The probability bound in \Cref{c:probboundalpha} gets worse when $\alpha$ is large, i.e., if there are more cherries.
On the other hand, we claim that there are significantly fewer witnesses with a large number of cherries.
In other words, the larger $\alpha$ is, the fewer witnesses $(W_j)_{1 \leq j \leq i}$ there are with $m(W_i)=\alpha sd$.
The following claim formalises this statement.
In fact, we will bound the number of choices for $W_i$ with $m(W_i)=\alpha s d$ for a given $(W_j)_{1 \leq j<i}$.

\begin{claim}\label{c:numberchoicesalpha}
  Given $i \in \{2, 3\}$, $\alpha \in \mathbb{R}$ and $(W_j)_{1 \leq j<i}$, let $s = (\gamma - 7K)^{i}/i!$.
  There are at most
  \[
    \exp\left( \left(\frac{1}{2}-\alpha\right)s \log d + \left(\alpha + \frac{1}{2}\right) s \log \log d + O(s)\right)
  \]
  witnesses $(W_j)_{1 \leq j \leq i}$ with $m(W_i)=\alpha s d$.
\end{claim}

\begin{proof}
  Let $\tau = d/(2 i \log d)$.
  We first fix integers $\ell, h$ and a set $H \subseteq S_0(x, i-1)$ of size $h$.
  Then we count the number of choices for $W_i \subseteq S_0(x,i)$ such that $|W_i| = s$, $m(W_i) = \alpha s d$,
  \[
    \{ v \in S_0(x, i-1)\, :\, |N(v) \cap W_i| \geq \tau \} = H, \]
  and $|N(H) \cap W_i| = \ell$.
  We start by noticing that every $v \in S_0(x, i)$ has at most $i$ neighbours in $W_{i-1}$ by \Cref{c:localconnectionstrong}\ref{i:backwardstrong}, and therefore
  \begin{equation}\label{eq:boundonh}
    h \leq \frac{is}{\tau}.
  \end{equation}
  The key idea to count the number of choices for $W_i$ given these parameters is as follows.
  We claim that imposing $m(W_i) = \alpha s d$ forces $\ell$ to be at least $2\alpha s - s/\log d$, i.e, this many elements of $W_i$ must come from $N(H)$ instead of being freely selected from $N(W_{i-1})$.
  This is a significant restriction if $\alpha \gg 1/\log d$: Since $|W_{i-1}| = \Theta(s/\gamma)$ while $h = O(s/\tau) = O(s/\gamma^2)$, a random subset of $N(W_{i-1})$ of size $s$ typically intersects $N(H)$ in $s/\gamma \ll \alpha s - s/\log d$ elements.

  To prove the bound on $\ell$, we start by noting that the cherries counted by $m(W_i)$ are paths of the form $uvw$ with $v \in S_0(x, i+1)$ and $u, w \in W_i$.
  For every such cherry, by \Cref{c:localconnectionstrong}\ref{i:cherries} there exists a $z=z(uvw) \in S_0(x, i-1)$ such that $uvwz$ forms a $C_4$ as in \Cref{f:witnessfori3}.
  Therefore, to count all cherries as above, we may pick a $z \in S_0(x, i-1)$ and two elements $u, w \in N(z) \cap W_i$.
  By \Cref{c:localconnectionstrong}\ref{i:cherrynew}, there is at most one $v \in S_0(x, i+1)$ which is a common neighbour of $u$ and $w$.
  Therefore,
  \begin{align*}
    m(W_i) &\leq \sum_{z \in S_0(x, i-1)} \binom{|N(z) \cap W_i|}{2}\\
           &= \sum_{z \in S_0(x, i-1) \setminus H} \binom{|N(z) \cap W_i|}{2} + \sum_{z \in H} \binom{|N(z) \cap W_i|}{2} \\
           &\leq \frac{\tau}{2} \sum_{z \in S_0(x, i-1) \setminus H} |N(z) \cap W_i| + \frac{d + (i-1)K}{2} \sum_{z \in H} |N(z) \cap W_i|,
  \end{align*}
  where in the last inequality we used that $|N(z) \cap W_i| \leq \tau$ for every $z \in H$ and, by \Cref{c:regular}, that $d(z) \leq d + (i-1)K$ for every $z \in S(x, i-1)$.
  Since $(i-1) K \leq \tau$, we may obtain
  \begin{equation}\label{e:cherries:firstbound}
    m(W_i) \leq \frac{\tau}{2} \sum_{z \in S(x, i-1)} |N(z) \cap W_i| + \frac{d}{2} \sum_{z \in H} |N(z) \cap W_i|.
  \end{equation}
  We proceed to bound the two sums.
  By \Cref{c:localconnectionstrong}\ref{i:backwardstrong}, every vertex of $W_i$ has at most $i$ neighbours in $S(x, i-1)$ and therefore
  \begin{equation}\label{eq:countneighboursz}
    \sum_{z \in S(x, i-1)} |N(z) \cap W_i| \leq is.
  \end{equation}
  For the second sum, double-counting and using that $n \leq 1+ \binom{n}{2}$ leads to
  \begin{equation}\label{eq:boundellfirst}
    \sum_{z \in H} |N(z) \cap W_i| = \sum_{w \in N(H) \cap W_i} |N(w) \cap H| \leq \sum_{w \in N(H) \cap W_i} \left(1 + \binom{|N(w) \cap H|}{2}\right).
  \end{equation}
  Observe that, by \Cref{c:localconnectionstrong}\ref{i:cherrynew}, each pair of elements $u, v \in H \subseteq S_0(x, i-1)$ can be contained in the neighbourhood of at most one $w \in W_i$.
  Therefore, the sum on the right-hand side of~\eqref{eq:boundellfirst} is at most $\ell + \binom{h}{2}$.
  Substituting~\eqref{eq:countneighboursz}, \eqref{eq:boundellfirst} and this bound into~\eqref{e:cherries:firstbound} leads to $m(W_i) \leq \frac{is\tau}{2} + \frac{d}{2}\left(\ell + \binom{h}{2}\right)$, and therefore
  \begin{equation}\label{e:ell:lower}
    \ell \geq \frac{2m(W_i) - is\tau}{d} - \binom{h}{2}.
  \end{equation}
  Recalling that $\tau = d / (2 i \log d)$, we may estimate
  \begin{equation}\label{eq:boundonparameters}
    \frac{is\tau}{d} = \frac{s}{2\log d} \quad \text{ and } \quad \binom{h}{2} \leq h^2 \leq \frac{i^2 s^2}{\tau^2} = o\left(\frac{s}{\log d}\right),
  \end{equation}
  since $i \leq 3$ and $s = \Theta\bigl(d^{i/2}/(\log d)^{i/2}\bigr) = o\bigl(\tau^2 / \log d\bigr)$.
  Setting $\ell_0 \coloneqq 2\alpha s - \frac{s}{\log d}$,~\eqref{e:ell:lower} and~\eqref{eq:boundonparameters} imply that $\ell \geq \ell_0$, as claimed.

  Having obtained a lower bound on $\ell$ (we may assume $\ell_0 \geq 0$), we may sum over all choices of $h, H, \ell, N(H) \cap W_i$ and $W_i \setminus N(H)$.
  Set $d' \coloneqq d+iK$, and recall that any $z \in B(x, i)$ has degree at most $d'$ by \Cref{c:regular}.
  We have that $H \subseteq S_0(x, i-1)$, where $|S_0(x, i-1)| \leq (d')^{i-1}$, that $|N(H)| \leq hd$ and that $W_i \subseteq N(W_{i-1}) \cap S_0(x, i)$, where $|N(W_{i-1}) \cap S_0(x, i)| \leq \gamma^{i-1} d'$.
  Therefore, the total number of choices for $W_i$ is at most
  \begin{equation}\label{eq:totalnumberchoices}
    \sum_{h=0}^{is/\tau} \binom{(d')^{i-1}}{h} \sum_{\ell=\ell_0}^s \binom{hd'}{\ell} \binom{\gamma^{i-1} d'}{s - \ell}.
  \end{equation}
    Using the inequality $\binom{a}{\ell} \leq (\frac{a}{b})^{\ell} \binom{b}{\ell}$, valid for all $1 \leq a \leq b$, we bound
    \[
    \binom{hd'}{\ell} \leq \left(\frac{h}{\gamma^{i-1}}\right)^{\ell} \binom{\gamma^{i-1}d'}{\ell}.
  \]
  The factor $\left(\frac{h}{\gamma^{i-1}}\right)^\ell$ is decreasing in $\ell$, since $\frac{h}{\gamma^{i-1}} = o(1)$ by \eqref{eq:boundonh} and the bound $s = \Theta(\gamma^i)$.
  Therefore, \eqref{eq:totalnumberchoices} is at most
  \begin{equation}\label{eq:boundonmsimplified}
    \sum_{h=0}^{is/\tau} \left(\frac{h}{\gamma^{i-1}}\right)^{\ell_0} \binom{(d')^{i-1}}{h} \sum_{\ell=0}^s \binom{\gamma^{i-1} d'}{\ell} \binom{\gamma^{i-1} d'}{s-\ell}.
  \end{equation}
  Using Vandermonde's identity $\sum_{\ell=0}^s \binom{a}{\ell} \binom{a}{s-\ell}= \binom{2a}{s}$, the bound in \eqref{eq:boundonmsimplified} is
  \begin{equation}\label{eq:vandermonde}
    \sum_{h=0}^{is/\tau} \left(\frac{h}{\gamma^{i-1}}\right)^{\ell_0} \binom{(d') ^{i-1}}{h} \binom{2 \gamma^{i-1} d'}{s}.
  \end{equation}
  Noting that the sum is monotone in $h$ and using the bounds $\binom{a}{b} \leq \left(\frac{ea}{b}\right)^b$ and $d' = (1+o(1))d$ we may upper bound \eqref{eq:vandermonde} further by
  \begin{equation}\label{eq:firstsum}
    \exp\left(- \ell_0 \log \frac{\tau \gamma^{i-1}}{is} +\frac{s}{2}(\log d + \log \log d)+ O(s)\right).
  \end{equation}
  Plugging in $\ell_0=2 \alpha s - \frac{s}{\log d}$, $\tau=\frac{d}{2i\log d}$ and $\gamma=\sqrt{\frac{d}{\log d}}$ into \eqref{eq:firstsum} yields that the number of choices for $W_i$ is bounded by
  \[
    \begin{split}
      &\exp\left(-\left(2 \alpha s - \frac{s}{\log d}\right)\left(\frac{1}{2} \log d - \frac{1}{2} \log \log d\right) +\frac{s}{2}(\log d + \log \log d)+ O(s)\right)\\
      &=\exp\left(\left(\frac{1}{2}-\alpha\right)s \log d + \left(\alpha+\frac{1}{2}\right) s \log \log d +O(s)\right),
    \end{split}
  \]
  as claimed.
\end{proof}

It only remains to combine the claims.
\begin{proof}[Proof of \Cref{l:xinAi+1-Ai}]\label{proof_xinAi+1-Ai}
  Fix $i \in \{2,3\}$ and suppose there is some $x \in \hat{A}_{i+1} \setminus \hat{A}_i$.
  By \Cref{c:existencewitness} there is a witness $(W_j)_{1 \leq j \leq i}$ satisfying \eqref{e:0-statement-dev}.
  Recall that every witness satisfies $W_j \subseteq S_0(x, j)$ and $|W_j|=\frac{(\gamma-7K)^j}{j!}$.
  Set $s = |W_i|$ and $d' = d+iK$.

  There are at most $\prod_{j=1}^{i-1} \binom{|W_{j-1}|d'}{|W_j|}$ choices for $(W_j)_{1 \leq j <i}$ (where we take $W_0 = \{x\}$), and there are $O(sd)$ possible values for $m$.
  Note that $O(sd)\cdot\prod_{j=1}^{i-1} (ejd'/(\gamma-7K))^{\Theta(\gamma^j)} = \exp(o(s \log\log d))$.
  Therefore, letting
  \[
    f(\alpha) \coloneqq \exp\Biggl(-\left(\frac{s}{2} - \frac{\alpha s}{1+2\alpha}\right)\left(\log d - 4\lambda \log \log d\right)+ \left(\frac{1}{2}-\alpha\right)s \log d + \left(\alpha +\frac{1}{2}\right)s \log \log d \Biggr),
  \]
  we observe by \Cref{c:probboundalpha}, \Cref{c:numberchoicesalpha} and the union bound that
  \[
    \Pr{x \in \hat{A}_{i+1}\setminus \hat{A}_i} \leq \exp(o(s \log\log d)) \cdot \max_\alpha f(\alpha).
  \]
  Therefore, it suffices to bound $f(\alpha)$.
  A straightforward calculation shows that
  \begin{equation}\label{e:f-bound}
    f(\alpha) = \exp\Biggl(\frac{1+4\lambda}{2} \cdot s \log\log d + \left(1 - \frac{4\lambda}{1+2\alpha} - \frac{2 \alpha \log d}{(1+2\alpha)\log \log d} \right) \alpha s \log\log d \Biggr).
  \end{equation}
  We claim that
  \[
    f(\alpha) \leq \exp\left(\frac{1+4\lambda}{2} \cdot s \log\log d + O(s)\right).
  \]
  Indeed, if the factor multiplying $\alpha s \log \log d$ in \eqref{e:f-bound} is negative, then the claim is trivial.
  Otherwise, we must have $\frac{2\alpha \log d}{\log \log d}< 1 + 2\alpha - 4\lambda$.
  Since $\alpha = O(1)$, this implies $\alpha = O(\log\log d / \log d)$ in this case.
  But then
  \[
    \left(1 - \frac{4\lambda}{1+2\alpha} - \frac{2 \alpha \log d}{(1+2\alpha)\log \log d} \right) \alpha s \log\log d = o(s),
  \]
  which also implies the claimed inequality.

  Hence, since $\lambda < -1/4$, the desired conclusion holds with $\beta = -\frac{1+4\lambda}{4} > 0$.
\end{proof}

\Cref{l:xinAi+1-Ai} bounds the probability that a vertex gets infected in the third or fourth step of the process.
Since we can explicitly calculate the probability that a vertex gets infected in the first round, it remains to prove an analogous bound for the second step.
It would be possible to adapt the statement and proof of \Cref{l:xinAi+1-Ai} to also cover this case, but the technical details become difficult to keep track of consistently.
Therefore, and since we can allow ourselves to be less careful here, we handle this case separately in the lemma below.

\begin{lemma}\label{l:xinA2-A1}
  Let $x \in V(G)$, $d=d(x)$, $\gamma=\gamma(x)$, $\lambda < -1/4$ and $p = \frac{1}{2}-\frac{1}{2} \sqrt{\frac{\log d}{d}} + \lambda \frac{\log \log d}{\sqrt{d \log d}}$.
  If there is a $K \in \mathbb{N}$ such that $G \in \mathcal{H}(K)$, then there is a $\beta >0$ (independent of $x$) such that
  \[
    \Pr{x \in \hat{A}_2 \setminus \hat{A}_1} \leq \exp(-\beta \gamma \log \log d).
  \]
\end{lemma}

\begin{proof}
  The proof strategy is similar to above, but we can allow ourselves to be less careful when counting the number of choices for witnesses.
  Suppose that there is a vertex $x \in \hat{A}_2\setminus \hat{A}_1$.
  Note that \Cref{c:existencewitness} also holds for $i=1$, and thus there is a witness $W_1 \subseteq N(x)$ with $|W_1|=s$ and
  \[
    e(W_1, S_0(x, 2) \cap \hat{A}_0) \geq p \cdot e(W_1, S_0(x, 2)) +t
  \]
  for $t=s\left(\frac{1}{2} \sqrt{d \log d} - \lambda \sqrt{\frac{d}{\log d}} \log \log d - 3 \gamma -9K\right)$.

  By \Cref{c:localconnectionstrong}\ref{i:backwardstrong} every vertex in $S_0(x, 2)$ has at most two neighbours in $W_1$.
  Let $X_1, X_2$ denote the set of vertices in $S_0(x, 2)$ that have exactly one or two neighbours in $W_1$, respectively.
  Note that $e(W_1, S_0(x, 2))= |X_1| + 2 |X_2|$.
  Furthermore, since any two vertices in $S_0(x, 1)$ have at most one common neighbour in $S_0(x, 2)$ by \Cref{c:localconnectionstrong}\ref{i:cherrynew} we can bound
  \[
    |X_2| \leq \binom{|W_1|}{2} =O(\gamma^2).
  \]
  This implies
  \[
    |X_1| + 4 |X_2| \leq \gamma d + O(\gamma^2) \leq \left(1+O\big((d \log d)^{-1/2}\big)\right) \gamma d.
  \]

  Again, since $e(W_1,S_0(x,2) \cap \hat{A}_0)$ is the sum of independent random variables of the form $j \cdot \Bin(|X_j|,p)$ from $j=1$ to $2$, we can apply Hoeffding's inequality (first inequality in \Cref{l:sumBin}), noting that $D(2) \leq \left(1+O\big((d \log d)^{-1/2}\big)\right) \gamma d$ and $\Ex{e(W_1, S_0(x, 2) \cap \hat{A}_0)}=p \cdot e(W_1, S_0(x, 2))$, to obtain
  \[
    \Pr{e(W_1, S_0(x, 2) \cap \hat{A}_0) \geq p \cdot e(W_1, S_0(x, 2)) +t} \leq \exp\left(-\frac{2t^2}{\left(1+O\big((d \log d)^{-1/2}\big)\right) \gamma d}\right).
  \]
  Plugging in the bound on $t^2$ from \eqref{eq:boundont2}, substituting $\gamma = s+7K$ and simplifying yields that this probability is at most
  \begin{equation}\label{eq:probA2-A1}
    \exp\left(-\left(\frac{s}{2} \log d -2 \lambda s \log \log d\right) + o(s \log \log d)\right).
  \end{equation}
  Moreover, since $W_1 \subseteq N(x)$ and $\frac{d}{s}= \Theta\left(\frac{d}{\gamma}\right)=\Theta(\sqrt{d \log d})$ there are at most
  \begin{equation}\label{eq:countW}
    \binom{d}{s} \leq \exp\left(\frac{s}{2} \log d + \frac{s}{2} \log \log d +O(s)\right)
  \end{equation}
  possible choices for $W_1$.
  Combining \eqref{eq:probA2-A1} and \eqref{eq:countW} we obtain
  \[
    \Pr{x \in \hat{A}_2 \setminus \hat{A}_1} \leq \exp\left(\left(\frac{1}{2}+2\lambda\right) s \log \log d + o(s \log \log d)\right) = \exp(-\Omega(s \log \log d)),
  \]
  where the last step is due to $\frac{1}{2}+2 \lambda <0$, which follows from $\lambda <-\frac{1}{4}$.
  Recalling that $s=\Theta(\gamma)$, we obtain \Cref{l:xinA2-A1}.
\end{proof}

Having these lemmas at hand, we can now prove the $0$-statement.

\begin{proof}[Proof of $0$-statement in \Cref{t:mainThm-newH}]
  Let $K \in \mathbb{N}$ and let $(G_n)_{n \in \mathbb{N}}$ be a sequence of graphs such that $G_n \in \mathcal{H}(K)$ and $\delta(G_n) \to \infty$ as $n \to \infty$.
  We write $G \coloneqq G_n$.

  Let $\lambda < -\frac{1}{4}$ and let
  \[
    p < \frac{1}{2}-\frac{1}{2} \sqrt{\frac{\log \dmin}{\dmin}} + \lambda \frac{\log \log \dmin}{\sqrt{\dmin \log \dmin}}.
  \]
  Since the function $d \mapsto \frac{1}{2}-\frac{1}{2}\sqrt{\frac{\log d}{d}} + \lambda \frac{\log \log d}{\sqrt{d \log d}}$ is increasing in $d$, by \Cref{l:xinAi+1-Ai} applied with $i=2$ we have that for every $x \in V(G)$
  \[
    \Pr[\big]{x \in \hat{A}_3 \setminus \hat{A}_2} = o(1).
  \]
  Similarly, it follows from \Cref{l:xinA2-A1} that for every $x \in V(G)$
  \[
    \Pr[\big]{x \in \hat{A}_2 \setminus \hat{A}_1} = o(1).
  \]
  Furthermore, since $\gamma(x) = \sqrt{\frac{d(x)}{\log d(x)}}=o\left(\sqrt{d(x) \log d(x)}\right)$, it is a simple consequence of Chernoff's inequality (\Cref{l:Chernoff}) that for every $x \in V(G)$
  \[
    \Pr[\big]{x \in \hat{A}_1 \setminus \hat{A}_0} \leq \Pr[\Big]{\Bin\left(d(x), p\right) \geq \frac{d(x)}{2}-3\gamma(x)} = o(1).
  \]
  Hence, by Markov's inequality whp $\big|\hat{A}_3 \setminus \hat{A}_0 \big| = o\left(|V(G)|\right)$.
  Using Chernoff's inequality again we obtain that whp $\big|\hat{A}_0\big|=\big|A_0\big| \leq \frac{3}{4} |V(G)|$, and hence whp
  \[
    \big|\hat{A}_3\big| = \big|\hat{A}_3 \setminus \hat{A}_0 \big| + \big|\hat{A}_0\big| < |V(G)|.
  \]
  On the other hand, by \Cref{l:xinAi+1-Ai} applied with $i=3$ there exists a $\beta >0$ such that for every $x \in V(G)$
  \[
    \Pr{x \in \hat{A}_4 \setminus \hat{A}_3} \leq \exp(-\beta \gamma(x)^3 \log \log d(x)).
  \]
  However, since $\gamma(x)^3 \log \log d(x)$ is increasing for $x$ sufficiently large, it follows that for every $x \in V(G)$
  \begin{equation}\label{e:finalprobbound}
    \Pr{x \in \hat{A}_4 \setminus \hat{A}_3} \leq \exp\left(-\beta \frac{\delta(G)^{3/2} \log \log \delta(G)}{\log^{3/2} \delta(G)}\right).
  \end{equation}

  Hence, by \eqref{e:finalprobbound} and \Cref{c:bound-larger}
  \[
    \Pr{\hat{A}_4 \neq \hat{A}_3} \leq |V(G)| \exp\left(-\beta \frac{\delta(G)^{3/2} \log \log \delta(G)}{\log^{3/2} \delta(G)}\right) = o(1).
  \]
  It follows that whp $\hat{A}_i = \hat{A}_3$ for all $i \geq 3$ and hence $A_i \subseteq \hat{A}_i = \hat{A}_3 \neq V(G)$ for all $i$.
  Therefore,
  \[
  \Phi(p,G) \coloneqq \Pr{\textbf{A}_p \text{ percolates on } G}= \Pr*{\bigcup_{i=0}^\infty A_i  = V(G)} = o(1),
  \]
completing the proof.
\end{proof}

\begin{remark}\label{rem:optimizing}
  It is perhaps natural to ask at this point if these bounds can be improved by a different choice of parameters.

  Firstly, one might try to consider the $Boot_k(\gamma)$-process for larger values of $k$.
  However, already in \cite{BaBoMo2009} it was shown that this does not lead to better bounds.
  In particular, we cannot prove \Cref{l:xinAi+1-Ai} for $i=4$ using these methods, since the bound on $h$ in \eqref{eq:boundonparameters} no longer holds.

  Secondly, we would be able to cover even larger graphs if we could increase the exponent of $\gamma=\sqrt{\frac{d}{\log d}}$ to a value strictly larger than $1/2$.
  However, as the second order term of the critical threshold is of that order, this would alter the infection process significantly.
\end{remark}

\section{The permutahedron --- \Cref{t:permutahedron}}\label{s:permuta}

In this section we show that the permutahedron lies in the graph class $\mathcal{H}$.
Let us first recall the definition of the permutahedron.
The \emph{$n$-dimensional permutahedron}, which we denote by $P_n$, is the graph whose vertices are the permutations of $[n+1]$, where two permutations $\sigma$ and $\pi$ are joined by an edge if and only if they differ by a transposition of two adjacent positions, i.e., if $\sigma = \pi\tau$ for some $\tau = (i, i+1)$ and $i \in [n]$.
For an illustration of $P_3$ and $P_4$, see \Cref{fig:permutahedron}.
Note that the permutahedron $P_n$ is $n$-regular and of order $(n+1)!$.
In particular, it does not lie in $\mathcal{F}$ as its order is superexponential in its degree.

\begin{figure}
  \centering
  \begin{minipage}{0.45\textwidth}
    \centering
    \begin{tikzpicture}
      \foreach \x/\name in {0/A, 60/B, 120/C, 180/D, 240/E, 300/F} {
        \coordinate (\name) at (\x:2cm);
      }

      \draw[thick] (A) -- (B) -- (C) -- (D) -- (E) -- (F) -- cycle;

      \foreach \name in {A,B,C,D,E,F} {
        \filldraw[black] (\name) circle (2pt);
      }

      \node at (-1,2) {$(1,2,3)$};
      \node at (1,2) {$(1,3,2)$};
      \node at (2.7,0) {$(2,3,1)$};
      \node at (1,-2) {$(3,2,1)$};
      \node at (-1,-2) {$(3,1,2)$};
      \node at (-2.7,0) {$(2,1,3)$};
    \end{tikzpicture}
  \end{minipage}\hfill
  \begin{minipage}{0.45\textwidth}
    \centering
    \begin{tikzpicture}
      \foreach [var=\x, var=\y, var=\z, count=\n] in {
        2/1/0,2/0/1,2/-1/0,2/0/-1,
        1/0/2,0/1/2,-1/0/2,0/-1/2,
        1/2/0,0/2/1,-1/2/0,0/2/-1,
        -2/1/0,-2/0/1,-2/-1/0,-2/0/-1,
        1/0/-2,0/1/-2,-1/0/-2,0/-1/-2,
        1/-2/0,0/-2/1,-1/-2/0,0/-2/-1
      }{\coordinate (n\n) at (\x,\y,\z);}
      \draw (n1)--(n2)--(n3)--(n4)--cycle;
      \draw (n5)--(n6)--(n7)--(n8)--cycle;
      \draw (n9)--(n10)--(n11)--(n12)--cycle;
      \draw (n21)--(n22)--(n23);
      \draw (n13)--(n14)--(n15);
      \draw (n6)--(n10);
      \draw (n2)--(n5);
      \draw (n8)--(n22);
      \draw (n15)--(n23);
      \draw (n7)--(n14);
      \draw (n11)--(n13);
      \draw (n1)--(n9);
      \draw (n3)--(n21);

      \draw[gray] (n17) -- (n18) -- (n19) -- (n20) -- cycle;
      \draw[gray] (n23) -- (n24) -- (n21);
      \draw[gray] (n12) -- (n18);
      \draw[gray] (n12) -- (n18);
      \draw[gray] (n20) -- (n24);
      \draw[gray] (n19) -- (n16);
      \draw[gray] (n17) -- (n4);
      \draw[gray] (n13) -- (n16) -- (n15);
    \end{tikzpicture}
  \end{minipage}
  \caption{$P_3$ and $P_4$}\label{fig:permutahedron}
\end{figure}
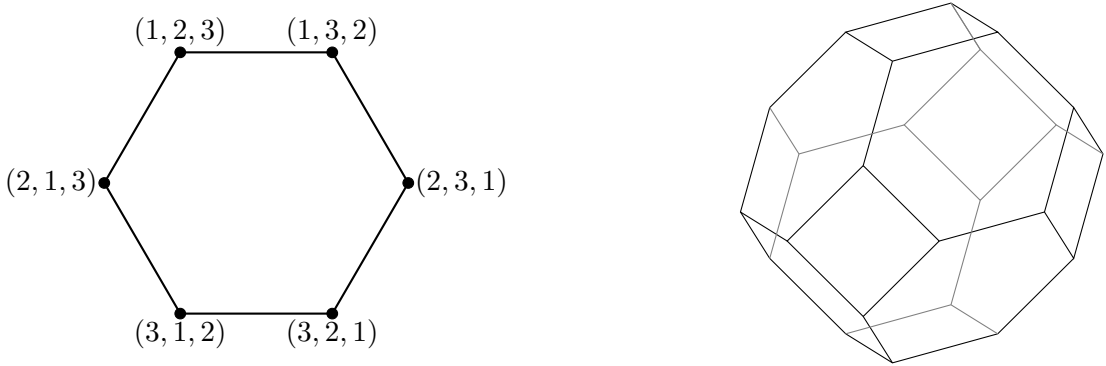

Here, we can think of the local coordinate system as being given by the $n$ adjacent transpositions $\{(i,i+1) \colon i \in [n]\}$.
When $n$ is large and $\ell$ is fixed, a \emph{typical} vertex $w$ at distance $\ell$ from $x$ can be identified uniquely by a set of $\ell$ transpositions which when applied to $x$ (in any order) result in $w$.
Of course, since the symmetric group $S_{n+1}$ is not abelian, this is not true for \emph{every} vertex $w \in S(x,\ell)$, but the set $D$ of non-typical vertices is relatively negligible in size and structure.

While most of the properties required for containment in $\mathcal{H}$ follow from this heuristic after some straightforward calculations, \Cref{c:projection} is slightly more subtle.
Here, a more geometric viewpoint can be helpful.
Let us consider the case of the hypercube as an analogy.
Here, viewing the hypercube as a polytope, the lower dimensional projections, in the sense of \Cref{c:projection} can be taken to be lower dimensional hypercubes, which are then the faces of this polytope.
In the case of the permutahedron, this suggests that the subgraphs in \Cref{c:projection} should be taken to be the $1$-skeleton of appropriately chosen faces of the permutahedron, considered as a polytope.

It can be shown that, unlike the hypercube, the lower dimensional faces of the permutahedron are not themselves lower dimensional permutahedra, but rather \emph{Cartesian products} of lower dimensional permutahedra (see for example \cite[Proposition 2.6]{P09}).
Hence, in order to show that $P_n$ is contained in $\mathcal{H}$, it will be necessary to show that this larger class of graphs lies in $\mathcal{H}$.
We note that a similar `projection' lemma was key to the analysis of the phase transition in bond percolation on the permutahedron in \cite[Section 4]{CDE24}, where again the notion of a projection had to be widened to include Cartesian products of permutahedra.

Let us define then $\mathcal{P}(n)$ to be the class of graphs that are a finite Cartesian product of permutahedra whose dimensions sum up to $n$, i.e., $G \in \mathcal{P}(n)$ if and only if $G= \car_{i=1}^m P_{n_i}$ with $\sum_{i=1}^m n_i =n$.
Note that each $G \in \mathcal{H}(n)$ is also $n$-regular.

We note that it can be shown that the $1$-skeleton of any zonotope, and so in particular every graph in $\mathcal{P}(n)$, is an isometric subgraph of some (possibly higher dimensional) hypercube, see \cite[Lemma 3.2]{CDE24}.
As is easily verified, and follows from \Cref{l:prodinH}, the $n$-dimensional hypercube satisfies
\[
  Q^n \in \mathcal{H}(2),
\]
which allows us to show that most of the properties required in the definition $\mathcal{H}$ hold for graphs in $\mathcal{P}(n)$ via this isometric embedding.
For completeness, and since we will need to explicitly refer to some properties of this isometry, we include a proof of this fact.

It will be convenient to think of the graphs in $\mathcal{P}(n)$ as Cayley graphs.
Indeed, given $k \in \mathbb{N}$ if we take the symmetric group $S_{n+k+1}$ and any subset $I \subseteq \{ (i,i+1) \,:\, 1 \leq i \leq n\}$ of adjacent transpositions of size $k$, then the Cayley graph of the subgroup of $S_{n+k+1}$ generated by $I$ will lie in $\mathcal{P}(n)$.
More explicitly, given some $I \subseteq \{ (i,i+1) \,:\, 1 \leq i \leq n\}$ we can decompose $I$ into maximal \emph{intervals}, sets of the form $\{(i,i+1) \,:\, i_0 \leq i \leq i_1\}$, say $I=I_1\cup \ldots I_m$ where $|I_j| = n_j$, in which case the Cayley graph generated by $I$ will be of the form $\car_{j=1}^m P_{n_j}$.

For a permutation $\pi \in S_{n+1}$ we say that a pair $\{a, b\} \subseteq [n+1]$ is an \emph{inversion} of $\pi$ if $(b-a) \cdot (\pi^{-1}(b) - \pi^{-1}(a)) < 0$, i.e., if the smallest element of the pair appears after the largest one in $\pi$.
Let $\Inv(\pi)$ denote the set of inversions of $\pi$.

Let $Q = Q^{\binom{n+1}{2}}$ denote the hypercube indexed by pairs $\{a, b\} \in \binom{[n+1]}{2}$, and let $f \colon V(P_n) \to Q$ be defined such that for $S = \{a, b\} \in \binom{[n+1]}{2}$ the value $f(\pi)_{S}$ equals $1$ if $\{a, b\}$ is an inversion of $\pi$ and $0$ otherwise (in other words, $f(\pi)$ is the characteristic vector of the set $\Inv(\pi)$ of inversions of $\pi$).
We claim that
\[
  \dist_{P_n}(\pi, \sigma) = |f(\pi) \oplus f(\sigma)| = d_{Q}(f(\pi), f(\sigma)),
\]
where $\oplus$ denotes coordinate-wise addition modulo $2$ (symmetric difference of the corresponding sets $\Inv(\pi)$ and $\Inv(\sigma)$), and so $f$ is an isometric embedding of the permutahedron $P_n$ into the $\binom{n+1}{2}$-dimensional hypercube $Q$.
Let us prove this claim.

Due to the symmetry of the Cayley graph structure, we have $\dist_{P_n}(\pi, \sigma) = \dist_{P_n}(\sigma^{-1}\pi, id) = |\Inv(\sigma^{-1} \pi)|$.
Therefore, it remains to show that $|\Inv(\sigma^{-1}\pi)| = |\Inv(\pi) \symdif \Inv(\sigma)|$.
Indeed, since $\pi^{-1}(i) = (\sigma^{-1}\pi)^{-1}(\sigma^{-1}(i))$ for every $i \in [n+1]$, we have that $\{\sigma^{-1}(a), \sigma^{-1}(b)\}$ is an inversion of $\sigma^{-1}\pi$ precisely when $(\sigma^{-1}(b) - \sigma^{-1}(a)) \cdot (\pi^{-1}(b) - \pi^{-1}(a)) < 0$, i.e., when $a$ and $b$ appear in different orders in $\sigma$ and $\pi$.

Finally, it is easy to verify that if $G_1$ and $G_2$ have isometric embeddings $f_1,f_2$ into $Q^{n_1}$ and $Q^{n_2}$ respectively, then the embedding $f \colon G_1 \car G_2 \to Q^{n_1 + n_2}$ given by $f(x,y) = (f_1(x),f_2(y))$ is also an isometric embedding.
In particular, since each graph in $\mathcal{P}(n)$ is the Cartesian product of permutahedra, one can derive explicit isometric embeddings of any graph in $\mathcal{P}(n)$ into a hypercube.

\begin{lemma}\label{l:permutahedronH}
  \[
    \mathcal{P}(n) \subseteq \mathcal{H}(4).
  \]
\end{lemma}

\begin{proof}
  We induct on $n$.
  Let $G \in \mathcal{P}(n)$, where by the discussion above we may without loss of generality assume that $G$ is the Cayley graph of a subgroup $H$ of some permutation group $S_{N+1}$ generated by some subset $I$ of the adjacent transpositions with $|I|=n$ which contains the identity.

  \begin{proof}[\Cref{c:regular}]\let\qed\relax
    $G$ is $n$-regular.
  \end{proof}

  \begin{proof}[\Cref{c:backwards}]\let\qed\relax
    This property is preserved under taking isometric subgraphs and is increasing in $K$. Since \Cref{c:backwards} holds in the hypercube $Q^n$ with $K=2$, it is also true in $G$ with $K=4$.
  \end{proof}

  \begin{proof}[\Cref{c:localconnectionstrong}]\let\qed\relax
    For a permutation $\sigma$ of $[n+1]$ let the support of $\sigma$ be the set $\supp(\sigma)=\{i \in [n+1] \,:\, \sigma(i)\neq i\}$.
    Given $x \in V(G)$ and $y \in S(x, \ell)$ for some $\ell \geq 1$, define
    \[
      T(x,y) = (\tau_1,\dots, \tau_{\ell})
    \]
    to be an arbitrary sequence of exactly $\ell$ adjacent transpositions in $I$ such that $y = x\tau_1 \dots \tau_\ell$.
    Set
    \[
      S_0(x, \ell)= \{y \in S(x,\ell) \,:\, \text{all transpositions in $T(x,y)$ have pairwise disjoint supports}\}.
    \]
    By definition $S_0(x,\ell) \subseteq S(x, \ell)$. Note that this definition of $S_0(x,\ell)$ implicitly defines the set
    \[
      D \coloneqq \bigcup_{\ell} S(x,\ell) \setminus S_0(x, \ell).
    \]
    Whenever $y \in S_0(x, \ell)$, the transpositions in the sequence $T(x, y)$ can be applied in any order.
    In this case, we will (in a slight abuse of notation) consider $T(x,y)$ as a subset of $I$, which satisfies $|T(x,y)| = \dist(x,y)=\ell$.

    If $w \in S(x,\ell) \cap D$, then $T(x, w) = (\tau_1, \ldots, \tau_\ell)$ is such that there exist $1 \leq i < j \leq \ell$ with $\supp(\tau_i)\cap\supp(\tau_j)\neq\emptyset$.
    Therefore,
    \[
      |S(x,\ell) \cap D| \leq 3 \binom{\ell}{2} n^{\ell-1},
    \]
    and so \ref{i:smallnew} holds.

    If $w \in S_0(x, \ell)$, then a neighbour of $w$ in $S(x, \ell+1) \cap D$ must differ from $w$ in a transposition that has non-disjoint support with one of the transpositions in $T(x, w)$, and hence there are at most $3 \ell$ of them.
    Note that any $w \in S_0(x, \ell)$ has no neighbours in $S(x, \ell-1) \cap D$, and no neighbours at all in $S(x,\ell)$, so \ref{i:sparsenew} holds.

    Furthermore, since properties \ref{i:cherrynew} and \ref{i:backwardstrong} hold in the hypercube $Q^n$, even when $K=2$ and $D=\emptyset$, and are preserved under taking isometric subgraphs and increasing in $K$, they are also satisfied in $G$ for our choice of $D$ and $K$.

    Finally, let $uvw$ be a cherry with $u, w \in S_0(x, \ell)$ and $v \in S_0(x, \ell+1)$.
    It is easy to verify that $|T(x, u) \cup T(x, w)|=\ell+1$ and $v= u \tau_1 = w \tau_2$ for some adjacent transpositions $\tau_1, \tau_2$ with disjoint support.
    Then $v\tau_1\tau_2$ is a common neighbour of $u$ and $w$ in $S_0(x, \ell-1)$, and so \ref{i:cherries} holds.
\end{proof}

\begin{proof}[\Cref{c:projection}]\let\qed\relax
  Without loss of generality, suppose that $x \in V(G)$ is the identity and let $\ell \in \mathbb{N}$ and $y \in S(x,\ell)$.
  Consider the set
  \[
    \Inv'(y) \coloneqq \{\{i, j\} \subseteq [n+1] \,:\, (i-j)\cdot(y(i)-y(j)) < 0\},
  \]
  that is, the set of pairs of \emph{positions} corresponding to inversions of $y$.
  Let $I(y) \subseteq I$ be the set of generators whose support intersects a pair in $\Inv'(y)$.
  Note that $|I(y)| \leq 4 \ell$.

  Let us consider the subgroup $H(y)$ of $S_{N+1}$ generated by all transpositions except those in $I(y)$ and its Cayley graph $\Gamma(y)$.
  The cosets of $H(y)$ in $\Gamma$ form a decomposition of the vertex set, giving rise to a natural packing of copies of $\Gamma(y)$ inside $G$.

  Let $G(y)$ be the copy of $\Gamma(y)$ which contains $y$.
  Clearly $y \in G(y) \subseteq G$.
  By construction, every vertex $z \in G(y)$ satisfies $\Inv'(y) \subseteq \Inv'(z)$ and hence $G(y) \cap B(x, \ell-1)=\emptyset$ and $G(y) \in \mathcal{P}(n')$ for some $n' \geq n-4\ell$.
  In particular, $G(y)$ is $n'$-regular, and for any $w \in V(G(y))$,
  \[
    |d_{G(y)}(w) - d_G(w)|  = n-n' \leq 4\ell.
  \]

  Finally, by the induction hypothesis, $G(y) \in \mathcal{H}(4)$.
\end{proof}

\begin{proof}[\Cref{c:partitioncond}]\let\qed\relax
  Given $x \in V(G)$, $\ell \in \mathbb{N}$ and $w \in S_0(x,\ell)$, let
  \[
    Z = \{ z \in S_0(x,\ell) \,:\, \dist(z,w) \leq 2\ell-1\}.
  \]
  By considering the isometric embedding of $G$ into the hypercube described above, it can be seen that for any $w,z \in S_0(x,\ell)$
  \[
    \dist(w, z) = |T(x, z) \triangle T(x,w)|=2|T(x,z) \setminus T(x,y)|.
  \]
  Hence, for every $z \in Z$,
  \[
    |T(x,z) \setminus T(x,w)| = \frac{1}{2} \dist(w, z) < \ell.
  \]
  It follows that
  \[
    Z \subseteq Z' = \{ z \in S_0(x, \ell) \,:\, |T(x,z) \setminus T(x,w)| \leq \ell-1\},
  \]
  and it is clear that
  \[
    |Z| \leq |Z'| \leq \sum_{i=0}^{\ell-1} \binom{n}{i} \leq 2n^{\ell-1}.
  \]
  \end{proof}

  \begin{proof}[\Cref{c:bound-larger}]\let\qed\relax
    Finally, if $G= \car_{i=1}^m P_{n_i}$ with $\sum_{i=1}^m n_i =n$, then $|V(G)| = \prod_{i=1}^m (n_i)! \leq n! \leq \exp( n \log n)$.
  \end{proof}
  This concludes the proof.
\end{proof}

Finally, we note that \Cref{l:permutahedronH} and \Cref{t:mainThm-newH} together imply \Cref{t:permutahedron}.

\section{The product of stars --- \Cref{t:star}}\label{s:productstars}

When the minimum and maximum degrees of the sequence $(G_n)_{n\in N}$ differ greatly, \Cref{t:mainThm-newH} does not even determine the second term in the expansion of the critical probability, which is sandwiched between a function of the minimum degree and a function of the maximum degree.
It remains a central question to determine for irregular high-dimensional graphs whether there is some \emph{universal} behaviour for irregular high-dimensional graphs, and if so, whether the second term in the expansion of the critical probability is quantitatively controlled by the maximum, minimum or average degree of, if by any of them.
In this section we give the first example of an irregular high-dimensional graph where we can determine this second term asymptotically, and show that in this case it coincides with the lower bound in Theorem \ref{t:mainThm-newH}, and so is controlled by the minimum degree.

Let $q \in \mathbb{N}$ be an arbitrary constant, and let $K_{1,q}$ denote the star on vertex set $\{0, 1, \ldots, q\}$ with centre $0$, so that the edges are all pairs of the form $\{0, i\}$ for $i \in [q]$.
We consider the graph $G_n = \car_{i=1}^n K_{1, q}$ obtained by taking the Cartesian product of $n$ copies of $K_{1,q}$.
We note that $\delta(G_n) = n$ and $\Delta(G_n) = qn$, and a simple calculation shows that the average degree is given by $d(G) = \frac{2qn}{q+1}$.
We will show that the expansion of the critical probability of $G$ has the form
\[
  \frac{1}{2} - \frac{1}{2}\sqrt{\frac{\log n}{n}} + O\left(\frac{\log \log n}{\sqrt{n \log n}}\right),
\]
which we note is independent of $q$.

First, let us show that $G_n$ indeed lies in $\mathcal{H}$.
Since this is true for all Cartesian product graphs with base graphs of bounded size, we prove it here for this more general graph class and obtain the statement for the Cartesian product of stars as a corollary.
\begin{lemma}\label{l:prodinH}
  Let $C \in \mathbb{N}$ be a fixed constant and let $H_1,\ldots, H_n$ be graphs such that $1< |V(H_i)| \leq C$ for all $i \in [n]$.
  Then $G_n \coloneqq \car_{i=1}^n H_i \in \mathcal{H}(C)$ for all $n\in \mathbb N$.
\end{lemma}

\begin{proof}
  By \cite[Lemma 8.2]{CEGK24}, each Cartesian product graph $\car_{i=1}^n H_i$ whose base graphs $H_i$ are of bounded size, i.e., $|V(H_i)| \leq C$, is in $\mathcal{F}(C)$.
  Thus, it remains to show Properties~\ref{c:localconnectionstrong}\ref{i:backwardstrong} and~\ref{c:localconnectionstrong}\ref{i:cherries} for $G_n$.

  As in \cite{CEGK24}, given distinct vertices $x,y \in V(G)$, let us define
  \[
    I(x,y) \coloneqq \{ i \in [n] \,:\, x_i \neq y_i\},
  \]
  noting that $1\leq|I(x,y)| \leq \dist(x,y)$.
  Given $x\in V(G)$ let
  \[
    D \coloneqq \{ y \in V(G) \,:\, |I(x,y)| \neq \dist(x,y)\},
  \]
  so that for all $\ell \in \mathbb{N}$
  \[
    S_0(x,\ell) \coloneqq S(x,\ell) \setminus D = \{y \in S(x,\ell) \,:\, |I(x,y)| =\ell\}.
  \]

  \begin{proof}[\Cref{c:localconnectionstrong}\ref{i:backwardstrong}]\let\qed\relax
    Let $v \in S_0(x, \ell)$.
    By definition $|I(x, v)|=\ell$, i.e., $v$ differs from $x$ in exactly $\ell$ coordinates.
    Thus, to reach a neighbour of $v$ in $S(x, \ell-1)$, one must change one of these coordinates to the value it assumes in $x$, which implies $|N(v) \cap S(x, \ell-1)|=\ell$.
  \end{proof}

  \begin{proof}[\Cref{c:localconnectionstrong}\ref{i:cherries}]\let\qed\relax
    If $u, w \in S_0(x,\ell)$ are distinct, and have a common neighbour $v \in S_0(x,\ell+1)$, then it is easy to verify that $|I(x,w) \cup I(x,u)| = \ell+1$, and $v$ is the unique vertex in $S_0(x,\ell+1)$ which agrees with $w$ on $I(x,w)$, with $u$ on $I(x,u)$ and with $x$ on $[n] \setminus (I(x,w) \cup I(x,u))$.
    Then, there is a vertex $z \in S_0(x, \ell-1)$ which agrees with $u$ and $w$ on $|I(x, w) \cap I(x, u)|$ and with $x$ on $[n] \setminus (I(x, w) \cap I(x, u))$.
    As $|I(x, w) \cap I(x, u)| = \ell-1$, $z$ is a common neighbour of $u,w$ in $S_0(x, \ell-1)$.
    Hence, $N(w) \cap N(u) \cap S_0(x, \ell-1) \neq \varnothing$.
  \end{proof}

  This concludes the proof.
\end{proof}

\begin{corollary}
  Let $q \in \mathbb{N}$ and let $G_n=\car_{i=1}^n K_{1,q}$.
  For all $n \in \mathbb{N}$
  \[
    G_n \in \mathcal{H}(q+1).
  \]
\end{corollary}

We continue with the proof of \Cref{t:star}.

\begin{proof}[Proof of~\Cref{t:star}]
  If $\lambda < -1/4$, then \Cref{t:mainThm-newH} readily implies that $\Phi(p,G_n) = o(1)$, so it remains to show the $1$-statement.

  So, suppose that $\lambda =\frac{1}{2}+\eps$ for some $\eps>0$.
  For $0 \leq i \leq n$, we denote by $L_i$ the set of vertices with exactly $i$ coordinates corresponding to the centre of a star, that is,
  \[
    L_i = \{ v \in V(G_n) \colon |\{k \in [n]: v_k = 0\}| = i\},
  \]
  which we call the \emph{$i$th layer} and we denote by $L_{[a, b]}$ the set $\bigcup_{i=a}^b L_i$.
  We note that, for every $v \in L_i$,
  \begin{equation}\label{e:degrees}
    \left|N^+(v)\right| \coloneqq |N(v) \cap L_{i+1}| = n-i \qquad \text{ and } \qquad \left|N^-(v)\right| \coloneqq |N(v) \cap L_{i-1}| = iq,
  \end{equation}
  so that $d(v) = iq + (n-i)$.
  We define
  \[
    i_0 = \lceil n^{5/6} \rceil \qquad\text{and}\qquad i_1 = \left\lfloor \frac{n}{q+1} \right\rfloor,
  \]
  noting that $|N^-(v)| \leq |N^+(v)|$ holds for $v \in L_i$ precisely when $i \leq i_1$.

  For all $v \in L_{[0,i_0]}$, $d(v) \leq n + O(n^{5/6})$.
  In particular, for each $v \in L_{[0,i_0]}$
  \[
    \sqrt{\frac{\log d(v)}{d(v)}}=\left(1-O\left(\frac{1}{n^{1/6}}\right)\right) \sqrt{\frac{\log n}{n}} = \sqrt{\frac{\log n}{n}} + o\left(\frac{\log \log n}{\sqrt{n \log n}}\right).
  \]
  Thus, it follows that
  \[
    p \geq \frac{1}{2} - \frac{1}{2} \sqrt{\frac{\log d(v)}{d(v)}} + \left(\frac{1}{2}+\frac{\eps}{2}\right) \frac{\log \log d(v)}{\sqrt{d(v) \log d(v)}},
  \]
  and so, by \Cref{l:super-exponential} there is a $\beta>0$ such that for all $v \in L_{[0,i_0]}$,
  \[
    \Pr*{v \notin A^{(11)}} \leq \exp(-\beta d(v)^2) \leq \exp(-\beta n^2).
  \]
  Since $|V(G_n)| = (q+1)^n = o\left( \exp(\beta n^2) \right)$, it follows that whp $L_{[0,i_0]} \subseteq A^{(11)}$.

  We aim to show that the infection `propagates upwards' through the layers (see \Cref{f:productstars}).
  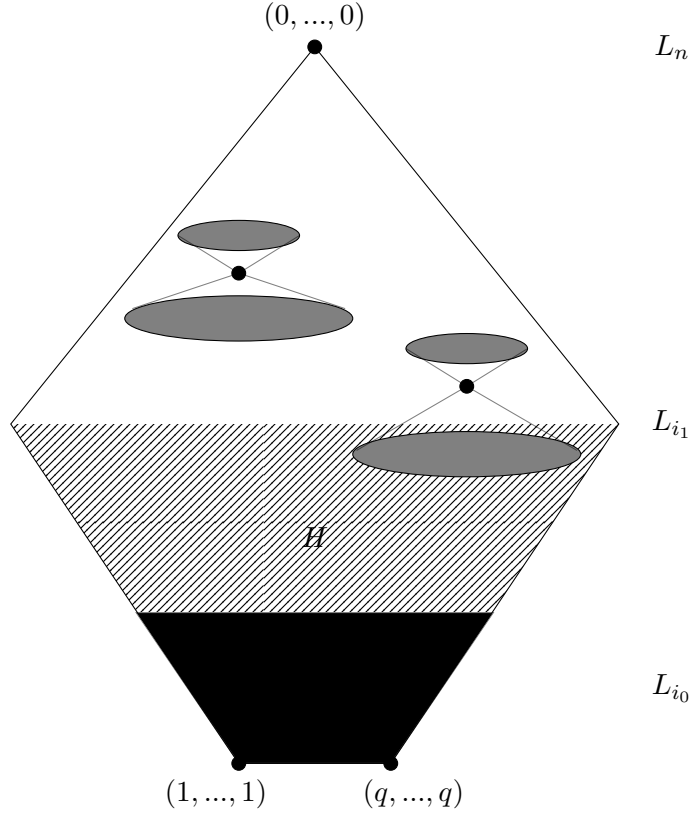
\begin{figure}
    \centering
    \begin{tikzpicture}
      \fill[pattern=north east lines] (-1,0) -- (1,0) -- (4,4.5) -- (3.9,4.5) -- (-3.9,4.5) -- (-4,4.5) -- cycle;
      \fill[draw=gray] (-1,0) -- (1,0) -- (2.35,2) -- (-2.35,2) -- cycle;
      \path[draw=black] (-1,0) -- (1,0) -- (4,4.5) -- (0,9.5) -- (-4,4.5) -- cycle;
      \node (C) at (0,9.5) [circle,draw, fill, scale=0.5] {};
      \node[] at (0,9.9) {$(0,...,0)$};
      \node (L1) at (-1,0) [circle,draw, fill, scale=0.5] {};
      \node[] at (-1.3,-0.4) {$(1,...,1)$};
      \node (L2) at (1,0) [circle,draw, fill, scale=0.5] {};
      \node[] at (1.3,-0.4) {$(q,...,q)$};

      \node[] at (4.7,9.5) {$L_n$};
      \node[] at (4.7,4.5) {$L_{i_1}$};
      \node[] at (4.7,1) {$L_{i_0}$};
      \node[] at (0,3) {$H$};
      \node[scale=2] at (-3.2,2.4) {};

      \node (L1) at (2,5) [circle,draw, fill, scale=0.5] {};
      \draw[fill=gray] (2, 4.1) ellipse (1.5cm and 0.3cm);
      \draw[fill=gray] (2, 5.5) ellipse (0.8cm and 0.2cm);
      \path[draw=black, draw opacity=0.5] (2,5) -- (0.55,4.15);
      \path[draw=black, draw opacity=0.5] (2,5) -- (3.45,4.15);
      \path[draw=black, draw opacity=0.5] (2,5) -- (1.2,5.5);
      \path[draw=black, draw opacity=0.5] (2,5) -- (2.8,5.5);

      \node (L1) at (-1,6.5) [circle,draw, fill, scale=0.5] {};
      \draw[fill=gray] (-1, 5.9) ellipse (1.5cm and 0.3cm);
      \draw[fill=gray] (-1, 7) ellipse (0.8cm and 0.2cm);
      \path[draw=black, draw opacity=0.5] (-1,6.5) -- (-2.4,6.03);
      \path[draw=black, draw opacity=0.5] (-1,6.5) -- (0.4,6.03);
      \path[draw=black, draw opacity=0.5] (-1,6.5) -- (-1.8,7);
      \path[draw=black, draw opacity=0.5] (-1,6.5) -- (-0.2,7);
    \end{tikzpicture}
    \caption{The black subgraph $H'$ consists of the lowest $i_0$ layers. The shaded subgraph $H$ takes up the rest of the lower half of the layer structure, i.e., all layers from $L_{i_0+1}$ to $L_{i_1}$. Every vertex that is not in $H \cup H'$ has more neighbours in the layer below it than in the layer above it.}\label{f:productstars}
  \end{figure}
  We let $H \coloneqq L_{[i_0, i_1+1]}$ and $\alpha=n^{-1/7}$.
  We claim that the set
  \[
    X \coloneqq \left\{w \in H : \left|N^+(w) \cap A^{(0)}\right| \leq \left(\frac{1}{2}-\alpha \right)\left|N^+(w)\right|\right\}
  \]
  is sparse, in the following sense.
  \begin{claim}\label{c:sparse}
    With high probability, every $v \in H$ satisfies $|N(v) \cap X| < d(v)^{1/3}$.
  \end{claim}
  \begin{proof}
    Fix $v \in H$.
    If $v$ does not satisfy the desired inequality, then there exists a set $R \subseteq N(v)$ such that $|R| = d(v)^{1/3}$ and $R \subseteq X$.
    Since $G_n$ is a Cartesian product of base graphs with co-degree one, it can be seen that for any two distinct vertices $w_1, w_2 \in N(v)$, $|N(w_1) \cap N(w_2)| \leq 2$.
    Furthermore, by \eqref{e:degrees}, for every $w \in H \cup N(H)$, $|N^+(w)| \geq (1/2+o(1))d(v)$.
    Hence, using \Cref{c:regular} we may bound
    \begin{align*}
      m\coloneqq\biggl|\bigcup_{w \in R} N^+(w)\biggr| \geq \sum_{w \in R}  |N^+(w)| - 2\binom{|R|}{2} \geq \sum_{w \in R}  |N^+(w)| - d(v)^{\frac{2}{3}} = (1+o(1)) \frac{d(v)^{\frac{4}{3}}}{2}.
    \end{align*}

    Furthermore, it also follows from the above that
    \[
      \sum_{w \in R} |N^+(w)|  - m \leq d(v)^{\frac{2}{3}}  = o\left( \alpha \sum_{w \in R} |N^+(w)| \right).
    \]

    Hence, by the Chernoff bounds, we can bound the probability of the event $R \subseteq X$ by
    \begin{align*}
      \Pr{R \subseteq X} &\leq \Pr[\bigg]{\bigg|A^{(0)} \cap \bigcup_{w \in R} N^+(w)\bigg| \leq \left(\frac{1}{2}-\alpha\right) \sum_{w \in R}\big| N^+(w)\big|}\\
                         & \leq \Pr[\bigg]{\Bin(m, p) \leq \left(\frac{1-\alpha}{2}\right) m}\\
                         & \leq \exp\left( - \Omega(\alpha^2 m)\right)= \exp\left(-\omega\left(d(v)\right)\right).
    \end{align*}
    Taking a union bound over all choices for $R \subseteq N(v)$ of size $d(v)^{\frac{1}{3}}$, we see that
    \[
      \Pr*{|N(v) \cap X| \geq d(v)^{\frac{1}{3}}} \leq \binom{d(v)}{d(v)^{\frac{1}{3}}} \exp\left(-\omega\left(d(v)\right)\right)=\exp\left(-\omega\left(d(v)\right)\right).
    \]
    Since $|V(G_n)| = \exp(O(\delta(G_n))$, a second union bound over all choices for $v$ concludes the proof of the claim.
  \end{proof}

  By the claim, we may assume that every vertex $v \in H$ has at most $d(v)^{1/3}$ neighbours in $X$.
  Moreover, as mentioned, \Cref{l:super-exponential} implies that $L_{[0,i_0]} \subseteq A^{(11)}$ with high probability.

  Now assume $L_{[0, i]} \setminus X \subseteq A^{(k)}$ for some $i_0 \leq i \leq i_1$ and some $k \in \mathbb{N}$, and let $v \in L_{i+1} \setminus X$.
  Since $v \not\in X$ and $N^-(v) \setminus X \subseteq A^{(k)}$, we have
  \[
    \left|N(v) \cap A^{(k)}\right| \geq \left(\frac{1}{2}-\alpha\right)\left|N^+(v)\right| + |N^-(v)| - d(v)^{1/3} \geq \frac{d(v)}{2},
  \]
  where we used that \eqref{e:degrees} and the fact that $i_0 \leq i$ implies that
  \[
    \frac{|N^-(v)|}{2} = \Omega(n^{5/6}) = \omega\left(  \alpha |N^+(v)| + d(v)^{1/3}\right).
  \]
  It follows that $L_{[0, i+1]} \setminus X \subseteq A^{(k+1)}$.
  Proceeding inductively, we see that every vertex of $L_{[0,i_1+1]} \setminus X$ is eventually infected, or more explicitly that $L_{[0,i_1+1]} \setminus X \subseteq A^{(i_1-i_0+12)}$.

  However, by \Cref{c:sparse}, any $v \in L_{[0,i_1]}$ has at most $d(v)^{1/3}$ neighbours in $X$, and hence at most $d(v)^{1/3}$ uninfected neighbours at round $i_1-i_0+12$.
  Hence, $L_{[0,i_1]} \subseteq A^{(i_1-i_0+13)}$.

  Finally, since by \eqref{e:degrees}, $|N^-(v)| > |N^+(v)|$ for every $v \in L_{[i_1+1, n]}$, it follows that $L_{[0,i_1]}$ is a percolating set.
  Hence, whp the initially infected set is percolating and $\Phi(p,G_n) = 1-o(1)$.
\end{proof}

\section{Discussion}\label{s:discussion}

In this paper we have analysed the location and width of the critical window for the random majority bootstrap percolation for a large family of high-dimensional graphs, in particular including graphs of superexponential order such as the permutahedron, strengthening the results of \cite{CEGK24} and generalising those of \cite{BaBoMo2009}.
It is natural to ask to what extent does this universality phenomenon hold, that is, what is the broadest class of high-dimensional graphs exhibiting this universal behaviour?
This question is discussed in more detail in \cite[Section 9]{CEGK24}.
It is perhaps particularly interesting to consider the graphs which lie in $\mathcal{F} \setminus \mathcal{H}$, such as the middle layers graph.
Here, the lack of $4$-cycles prevent analysis as in \Cref{sec:H'}, and it would be interesting to know if the order of the third term in the expansion of the critical probability can be bounded as in \Cref{t:mainThm-newH} via different methods, or if the large girth here results in a quantitatively different threshold.

For regular graphs in $\mathcal{H}$, \Cref{t:mainThm-newH} determines the first two terms in the expansion of the critical probability, and gives explicit bounds for the order of the third term.
A natural next step would be to determine the third term.
However, even in the `simplest' case of the hypercube this remains open.
\begin{question}
  Does there exist $\lambda^* \in [-2,1/2]$ such that if
  \[
    p \coloneqq \frac{1}{2}-\frac{1}{2} \sqrt{\frac{\log n}{n}} + \lambda \frac{\log \log n}{\sqrt{n \log n}},
  \]
  then
  \[
    \lim_{n \to \infty} \Phi(p, Q^n) = \left\{
      \begin{array} {c@{\quad \textup{if} \quad}l}
        0 & \lambda < \lambda^*,       \\[+1ex]
        1 & \lambda > \lambda^*.
      \end{array}\right.
  \]
\end{question}
It is conjectured in \cite{BaBoMo2009} that $\lambda^*$ exists and is equal to $1/2$.

We have also given the first example of an \emph{irregular} graph for which the second term of the expansion of the critical probability can be determined, where here the quantitive aspects are controlled by its minimum degree.
It is not clear if a similar argument can be made for arbitrary product graphs, where it is not true in general that all high-degree vertices have many neighbours with lower degree -- consider for example a product graph in which each factor is a clique joined to a long path.
In general, it is an intriguing question whether there is some universal behaviour for irregular product graphs, and specifically whether the second term in the critical probability is always controlled by the minimum degree.

\begin{question}
  Let $\varepsilon >0$ be sufficiently small.
  Does there exist a high-dimensional product graph $G$ such that $\lim_{n \to \infty} \Phi(p, G_n) = 0$ for some
  \[
    \frac{1}{2} - \left(\frac{1}{2} - \varepsilon\right) \sqrt{\frac{\log \delta(G)}{\delta(G)}} < p < \frac{1}{2} - \frac{1}{2} \sqrt{\frac{\log \Delta(G)}{\Delta(G)}}?
  \]
\end{question}

In another direction, threshold behaviour has been well-studied in the $r$-neighbour bootstrap percolation process for small constant $r$.
In particular, for $r=2$ a threshold has been determined on the hypercube and other high-dimensional grids \cite{BaBo2006HypercubeR2,BaBoMo2010HighDimGrid}.
It would be interesting to determine a similar threshold in other high-dimensional graphs, for example in the case of the middle layers graph or the permutahedron.
We also note that, even for the hypercube $Q^n$, very little is known about the typical behaviour of the process for fixed $r \geq 3$, see for example \cite[Conjecture 6.3]{BaBoMo2010HighDimGrid}.

\subsection*{Acknowledgements}
This research was funded in whole or in part by the Austrian Science Fund (FWF) [10.55776/DOC183, 10.55776/P36131].
For open access purposes, the author has applied a CC BY public copyright license to any author accepted manuscript version arising from this submission.

\printbibliography

\end{document}